\documentclass[11pt]{amsart} 
\usepackage{geometry}
\usepackage{mystyle}
\newcommand{\M}{\mathcal{M}}
\newcommand{\B}{\mathcal{B}}
\newcommand{\Y}{\mathbf{Y}}

\newcommand{\fl}{\mbox{flat}}
\renewcommand{\k}{\mathcal{K}}
\newcommand{\A}{\mathcal{A}}
\newcommand{\rainbow}{~\mbox{\scriptsize{rainbow}}}
\newcommand{\f}{{f_{r}}}

\title{A Local Limit Theorem for Cliques in $G(n,p)$}
\author{Ross Berkowitz}
\date{} 
\begin{document}
\maketitle
\begin{abstract}
We prove a local limit theorem the number of $r$-cliques in $G(n,p)$ for $p\in(0,1)$ and $r\ge 3$ fixed constants.  Our bounds hold in both the $\ell^\infty$ and $\ell^1$ metric.  The main work of the paper is an estimate for the characteristic function of this random variable.  This is accomplished by introducing a new technique for bounding the characteristic 
function of constant degree polynomials in independent Bernoulli random variables, combined with a decoupling argument.
\end{abstract}

\section{Introduction}
In 1960 Erd\H{o}s and R\'enyi introduced the study of $G(n,p)$, the random graph on $n$ vertices where each edge is included independently at random with probability $p$.  In  \cite{ErdosRenyi} they showed, among other results, that the number of cliques of size $r$ in $G(n,p)$ is concentrated about its mean using Chebyshev's inqeuality.  Since then the Erd\H{o}s-R\'enyi random graph has become an object of much study, and many nice results have been obtained concerning the following natural question:
\begin{question}
Let $H$ be some fixed graph.  What is the distribution of the number of copies of $H$ as a random variable?
\end{question}

In this paper, we will consider this question for the regime where $H$ is the $r$-clique, $K_r$, and $p\in (0,1)$ is a fixed constant.
Let $\f$ denote the random variable counting the number of $r$-cliques in $G(n,p)$ and set $\mu=\E[f_r]$ and $\sigma^2=Var(f_r)$. 


In the 1980's there were several papers studying which subgraph counts obeyed a central limit theorem  (see \cite{Kar83, Kar84, NW88, Ruc88}, for example).   By that time central limit theorems stating that $\f$ converged in distribution to the Gaussian were known.  That is for any real numbers $a<b$ 
\begin{equation}\label{clteq}
\Pr\left[f_r\in[\mu+a\sigma, \mu+b\sigma]\right]=\frac{1}{\sqrt{2\pi}}\int_{a}^b e^{-t^2/2}dt+o(1)
\end{equation}
%

Note that the central limit theorem in equation \ref{clteq}
bounds the probability that $\f$ lies in an interval of length $O(\sigma)$.  In this paper we will show that the distribution of $\f$ is \textit{pointwise} close to a discrete Gaussian.   Our main result is the following local limit theorem:

\begin{thm} \label{Sup Main}
Fix any $0<\tau<\min(1/12,1/2r)$.
For any $m\in \mathbb{N}$ we have that
$$\Pr[f_r=m]=\frac{1}{\sqrt{2\pi}\sigma}e^{-\frac{\left(m-\mu\right)^2}{2\sigma^2}}+O\left(\frac{1}{\sigma n^{\frac12-\tau}}\right)$$
\end{thm}

Because of the quantitative error bound, we are also able to extend this to the following $\ell^1$, or statistical distance bound between 
$\f$ and the discrete Gaussian.
\begin{thm}\label{L1 Main}
$$\sum_{m\in \mathbb{N}}\left|\Pr(\f=m)-\frac{1}{\sqrt{2\pi}\sigma}\exp\left(-\frac{\left(m-\mu\right)^2}{2\sigma^2}\right)\right|=O(n^{-\frac12+\tau})$$
\end{thm}

\subsection{Related Results}
Our methods depend on examining a particular orthogonal basis of the space of functions on $G(n,p)$.  For many other applications of such orthogonal decompositions to counting problems on $G(n,p)$, see \cite{MR1219708}.
If $p$ were allowed to become arbitrarily small as $n$ grows, then $\f$ may be shown in some cases to resemble a Poisson random variable.  For example, if the edge probability $p\sim cn$ for some constant $c$, then Erd\H{o}s and Renyi \cite{ErdosRenyi} showed that the number of triangles in $G(n,p)$ converges to a Poisson distribution.  This result was a local limit theorem, as it estimated the pointwise probabilities $\Pr[f_3=k]$ for $k$ constant.  Further, R\"{o}llin and Ross \cite{Ross} showed a local limit theorem when $p\sim cn^{\alpha}$ for $\alpha\in [-1,-\frac12]$.   In this regime they showed that the triangle counting distribution converges to a translated Poisson distribution (which is in turn close to a discrete Gaussian) in both the $\ell^\infty$ and $\ell^1$ metrics.
In 2014, Gilmer and Kopparty \cite{JustinTriangles} proved a local limit theorem for triangle counts in $G(n,p)$  in the regime where $p$ is a fixed constant.  Their main theorem was the following pointwise bound:
$$\Pr[f_3=m]=\frac{1}{\sqrt{2\pi}\sigma}\exp\left(-\frac{\left(m-\mu\right)^2}{2\sigma_n^2}\right)\pm o(n^{-2})$$
The proof in \cite{JustinTriangles} proceeded by using the characteristic function.  The main step there was to show that $|\varphi(t)-\varphi_{f_3}(t)|$ is small for $t\in [-\pi \sigma_n,~\pi\sigma_n]$, where $\varphi$ represents the characteristic function of the standard normal distribution, and $\varphi_{f_3}$ represents the characteristic function the triangle counting function $f_3$.
\cite{triangles} extended this result by improving the error bound and obtained a bound on the statistical distance between $f_3$ and the discrete Gaussian as well.  

\subsection{High Level Overview of Techniques}

  The central technique in this paper is to examine the characteristic function $\varphi_{\k}(t)$, where $\k=(f_r-\mu)/\sigma$ is the
  mean 0 variance 1 normalization of $f_r$.  The main calculation is showing that 
  $$\int_{t=-\pi\sigma}^{\pi\sigma} \left|\varphi_\k(t)-\varphi(t)\right|dt=O(n^{-\frac12+\tau})$$
  where $\varphi(t):=e^{-t^2/2}$ is the characteristic function of the standard unit normal random variable.
  The local limit theorem then follows from Fourier inversion for lattices.
  However, bounding the characteristic function of sums of dependent random variables is a tricky problem and several new ideas were needed.
  
  \subsubsection{Estimating $\varphi_\k(t)$ for small $t$}
  First, building on the method in our earlier work \cite{triangles}, we rewrite our random variable $f_r$ as a polynomial, not in the natural 0,1 indicator random variables $x_e$, but instead in the orthogonal $p$-biased Fourier basis $\chi_e$.  This
  slight change of basis immediately simplifies the proof of the central limit theorem and lays bare the intuition that the number of triangles in $G(n,p)$
  is almost completley driven by the \textit{number of edges} present in the graph.  In fact, once we switch from $x_e$ to $\chi_e$ and normalize to unit variance, the degree
  ${r\choose 2}$ polynomial $\k=(\f-\mu)/\sigma$ becomes $1-o(1)$ close to a degree 1 polynomial.  This turns out to be sufficient to prove that $
  \varphi_\k(t)$ is close to a Gaussian for $t$ small.  Because $|e^{itx}-e^{itx'}|\le |x-x'|$ for any $x,x'\in \mathbb{R}$ we can simply estimate $\varphi_\k(t)$ by noting that 
  \begin{equation}\label{introsmallteq}
 | \E[e^{it\k}]-\E[e^{it\k^{=1}}]|\le \E[|t\k^{>1}|]
  \end{equation}
  
  Because $\k^{=1}$ is a sum of i.i.d.\ Bernoulli random variables, the fact that its characteristic function is close to Gaussian is the well
  known Berry-Esseen bound.  Meanwhile, we will show that as noted above, $\k$ is concentrated on degree 1 terms and so $\k^{>1}$ is small.
  
  \subsubsection{Bounding $\varphi_\k(t)$ for slightly larger $t$}
  The bound in equation \ref{introsmallteq} is useful, but crude, and it degrades in usefulness rapidly as $t$ grows.    Let $X=\k^{=1}=\sum_{e} \hat\k(e)\chi_e$ and $Y=\k^{>1}$ (recall that we will expect $Y$ to be small).  Then $\k=X+Y$ and we obtain a better approach by using Taylor's Theorem to rewrite the above estimate as 
    \begin{align*}\label{introsmallteq}
 | \E[e^{it\k}]-\E[e^{itX}]|=|\E[e^{itX}(e^{itY}-1)]|=\E\left[e^{itX}\sum_{j=1}^\ell \frac{(itY)^j}{j!}\right]+O\left(\E\left|e^{itX}(tY)^{\ell+1}\right|\right)
  \end{align*}
  
Assuming $tY$ is typically small and $\ell$ some large but fixed constant we will be able to show that $e^{itX}$ and $Y^j$ are nearly uncorrelated.
To this end we prove a result which, with some omitted terminology, says:
  
  \begin{thm}\label{mainchf}
  Let $Z=X+Y$, where $X=\sum_{i=1}^n X_i$ is an i.i.d.\ sum of Bernoulli random variables.  Assume that $\hat{\|}Y\hat{\|}_1=O(n)$,
  and $\varphi_{X}(t)=O(\exp(-n^{\Omega(1)}))$.
  Then for any fixed $\ell$
  $$|\varphi_{Z}(t)-\varphi_X(t)|=O\left(\left|t\cdot\|Y\|_2\right|^\ell\right)$$
  \end{thm}

  \subsubsection{Bounding $\varphi_\k(t)$ for $t$ even larger still}
  Several substantial barriers present themselves for adapting the above arguments to bounding $\varphi_\k(t)$ for $t\ge O(n)$.
  First, in order to apply Theorem \ref{mainchf} profitably, there was the requirement that we consider a random variable of the form $X+Y$ where
  $X=\sum X_i$ is a sum of i.i.d.\ random variables, and $t\|Y\|_2$ is small.   This is a source of trouble as once $t>\|Y\|_2^{-1}$ our bound will be worthless.  Second, and equally troubling, the characteristic function the sum of ${n\choose 2}$ i.i.d.\ independent Bernoullis, $\sum \frac{1}{n}\chi_e$  is only small for $t=O(n)$, but we require our characteristic function to be small for $t\le \sigma=O(n^{r-1})$.  It should be noted that this
  barrier is not artificial.  Some subgraph counts, such as the number of disjoint pairs of edges, do obey a central limit theorem by the proofs above,
  but not a local limit theorem.  In these cases the problem occurs because of the breakdown of the characteristic function at 
  $t=O(n)$  Again, this is not accidental, but a consequence of the fact that the number of pairs of disjoint edges is always a square, and therefore
  almost on a lattice of step size $O(n)$.
  
  The main idea is, very roughly speaking, that the higher order terms of the polynomial $\k$ are responsible for controlling the size of
  $\varphi_{\k}(t)$ for $t$ large.  In particular, when $t=n$, it is most profitable to look at $\k^{=2}$, the degree 2 polynomial rather than
  the $\k^{=1}$ as in the previous arguments.  However there is still trouble:  what to do with the larger $\k^{=1}$ term?  The answer
  lies in a decoupling trick which allows us to ``clear out'' the lower order terms.  We illustrate with an example extracted from \cite{JustinTriangles}.
  \begin{example}\label{cs example}
  Let $f_3$ be the triangle counting random variable.  Partition the vertex set $[n]=U_0\cup U_1$ with
  $|U_0|=|U_1|=n/2$.  Let $B_0$ denote the edges internal to $U_0$, and $B_1$ be all other edges.  Let $X\in \{0,1\}^{B_0}$ and $Y\in \{0,1\}^{B_1}$
  be random vectors drawn according to the probability distribution $G(n,p)$.  Finally let $Y_0,Y_1$ denote independently drawn copies of $Y$.  Finally
  rewrite $f_3=A(X)+B(Y)+C(X,Y)$, isolating the monomials in $f_3$ which only depend on either $X$ or $Y$.  Then we can bound the characteristic function of $f_3$ by doing the following decoupling trick
\begin{align*}
|\varphi_{f_3}(t)|^2&=|\E_{X,Y}[e^{itf_3}]|^2=\left| \E_X e^{it A(X)} \E_Y e^{it(B(Y)+C(X,Y))}\right|^2\le \E_X\left| \E_Y e^{it(B+C)}\right|^2\\
&=\E_X\left(\E_{Y_1}e^{it(B+C)}\overline{\E_{Y_2}e^{it(B+C)}}\right)=\E_{Y_1,Y_2} e^{it(B(Y_0)-B(Y_1))}\E_Xe^{it(C(X,Y_0)-C(X,Y_1))}\\
&\le\E_{Y_0,Y_1}\left |\E_X e^{it(C(X,Y_0)-C(X,Y_1))}\right|
\end{align*}

In the last line above,  the terms $A$ and $B$, which depended on only one of $X$ or $Y$, have vanished.  Additionally, in the inner expectation we consider $C(X,Y_0)-C(X,Y_1)$ as a polynomial in $X$ for some random but fixed choice of $Y_0,Y_1$.  
A moment's reflection will reveal that the only monomials in $C(X,Y)$ correspond to triangles with two vertices in $U_0$ and one vertex in $U_1$.  Therefore
each triangle represented in $C(X,Y)$ has two edges in $B_1$, but only one in $B_0$.  So $C(X,Y)$ is only a polynomial of
degree 1 in $X$. 

Then we can use standard methods to analyze $\E[e^{it[C(X,Y_0)-C(X,Y_1)]}]$, because it is a sum of independent Bernoulli random variables.
One last wrinkle in the above that should be mentioned is that the linear function $C(X,Y_0)-C(X,Y_1)$ depends on the samples $Y_0,Y_1$ of edges
in $B_1$.  But after some work, we can show that with overwhelming probability (in the sampling of $Y_0,Y_1$), we will have that
$\E[e^{it[C(X,Y_0)-C(X,Y_1)]}]$ is small.

\end{example}

Section \ref{decoupling section} develops a version of this decoupling trick for higher degree polynomials.  In order to eliminate all monomials of degree
at most $k-1$, we will require $k+1$ partitions of our vertices and $2k$ independent samples.   One additional difference will be that, upon performing 
this decoupling trick, we will not always be left with a linear function but rather a polynomial which is highly concentrated on degree 1 terms.
But combining some careful analysis with Theorem \ref{mainchf}, we will be able to obtain our bounds on $\varphi_{\k}(t)$ in a similar manner
to the above example.

\subsection{Organization of this Paper}
In Section \ref{prelimsection} we set up our notation and introduce some facts which will be necessary for the later sections.  Section \ref{mainsection} contains the statements and proofs of our main results, modulo the main technical lemmas.  In Section \ref{mainchf section} we prove Theorem \ref{mainchf}, which is our main technical tool for bounding characteristic 
functions of constant degree polynomials in this paper.  In Section \ref{decoupling section} we prove our main decoupling Lemma.  Section \ref{KrProp}
contains our analysis of the properites of the clique counting random variables $f_r$ and $\k$.
Finally, Sections \ref{smallt section} through \ref{larget section} are dedicated to applying the afforementioned Lemma to bounding the characteristic
function of $\k$ in different regimes depending on $|t|$.

%

\section{Preliminaries and Notation} \label{prelimsection}
\subsection{Definition of our random variables $f_r$ and $\k$}
Throughout we will always be working with the probability space $G(n,p)$.  We will assume a vertex set of $[n]=\{1,2,\ldots,n\}$, and 
a set of indicator random variables, $x_e$ for each edge $e\in {[n]\choose 2}$.  $x_e$ will be 1 if edge $e$ is present in our sampled graph and 0 otherwise.  All edges will be present independently at random with probability $p$.  We will use $\lambda$ to denote $\min(p,1-p)$.
The graph $K_r$ is the clique on $r$ vertices, that is it has $r$ vertices and contains all edges between them.
Let $f_r$ denote the random variable counting the number of copies of $K_r$ in our random graph.  We express this as
$$f_r=\sum_{S\equiv K_r} x^S$$
where the sum is taken over all  ${n\choose r}$  sets of edges $S\subset {[n]\choose 2}$ which are isomorphic to the $r$-clique $K_r$, and $x^S:=\prod_{i\in S} x_i$.  We will also
frequently refer to the mean and standard deviation of $f_r$.  Throught the paper we will use $\mu$ and $\sigma$ to denote
$$\mu:=\E_{G\sim G(n,p)} f_r(G)=p^{r\choose 2}{n\choose r}\qquad\qquad \sigma:=\sqrt{\E_{G\sim G(n,p)}[f_r^2-\mu^2]}$$
Note that $\mu$ and $\sigma$ depend on $n$, as well as the fixed parameters $r$ and $p$.  Throughout it will be more convenient to work with the normalized copy of $f_r$, which we label $\k$
$$\k:=\k_r(G):=\frac{f_r-\mu}{\sigma}$$

\subsection{Parameters and asymptotics}
We will have need of a fixed, but arbitrarily small constant labeled $\tau \in (0,1/2r)$.  $\tau$ will be the same constant throughout the entire paper.  Additionally, we will always assume that $r\ge 3$, and $p\in (0,1)$ are fixed constants which do not depend on $n$.  Our results
will then apply in the asymptotic setting as $n\to \infty$.  Additionally, all asymptotic notation of the form $O(\cdot),~o(\cdot),$ or $\Omega(\cdot)$ will view
$r,p,\tau$ as fixed constants.
\subsection{$p$-biased basis}
A crucial tool throughout this paper will be the $p$-biased Fourier basis.  Rather than working with the indicator random variable $x_e$ directly
we will instead apply the following linear transformation
\begin{define}[Fourier Basis]
$$\chi_e:=\chi_e(x_e)=\frac{x_e-p}{\sqrt{p(1-p)}}=\begin{cases}
-\sqrt{\frac{p}{1-p}}&\mbox{if }x=0\\
\sqrt{\frac{1-p}{p}}&\mbox{if }x=1
\end{cases}
$$
For $S\subset {[n]\choose 2}$ set $\chi_S=\prod_{e\in S}\chi_e$.  For $S=\varnothing$ we have $\chi_\varnothing\equiv 1$.
\end{define}
As the $x_e$ variables are independent $p$-biased Bernoulli random variables, for any $S\neq T\subset {[n]\choose 2}$ we have $\E[\chi_S^2]=1$ and $\E[\chi_S\chi_T]=0$.  Therefore
the set of functions $\{\chi_S~|~S\subset {[n]\choose 2}\}$ is an orthonormal basis for the space of functions on $G(n,p)$.

We will also use the notation of the Fourier transform.  Since the $\chi_S$ form an orthonormal basis, for any function 
$f:\{0,1\}^{[n]\choose 2}\to \mathbb{R}$ we can choose coefficients $\hat{f}(S)$ so that
$$f= \sum_{S\subset {[n]\choose 2}} \hat{f}(S)\chi_S$$
The coefficients $\hat{f}(S)$ are called the Fourier Coefficients of $f$, and the function $\hat{f}:2^{[n]\choose 2}\to \mathbb{R}$ is called
the Fourier transform of $f$.  Morevoer we can compute these coefficients by noting that $\hat{f}(S)=\E[f\chi_S]$.  The degree of a monomial $\chi_S$ is $|S|$, and for an arbitrary function $f$, we say that it has degree equal the degree
of the largest monomial in its Fourier expansion.  That is
$\deg(f)=\max_{\hat{f}(S)\neq \varnothing}|S|$.

Additionally, it will be helpful to refer only to terms of $f$ of a certain degree.  For any $k\in \mathbb{N}$ let
\begin{align*}
f^{=k}&:=\sum_{|S|=k}\hat{f}(S)\chi_S\\
f^{>k}&:=\sum_{|S|>k}\hat{f}(S)\chi_S
\end{align*}

The $2$-norm of a function $\|f\|_2$ and the spectral 1-norm $\hat{\|}f\hat{\|}_1$ are defined to be
\begin{align*}
\|f\|_2&:=\sqrt{\E[f^2]}\\
\hat{\|}f\hat{\|}_1&:=\sum_S |\hat{f}(S)|
\end{align*}
Another fact which we will need throughout this paper is Parseval's Theorem, which allows us to easily compute the variance of a function 
in terms of its Fourier transform
\begin{thm}(Parseval's Theorem) \label{parseval}
$$\E[f^2]=\sum_{S\subset {[n]\choose 2}} \hat{f}(S)^2$$
Also, it therefore holds that
$$Var(f)=\E[f^2]-\E[f]^2=\E[f^2]-\hat{f}(\varnothing)^2=\sum_{S\neq \varnothing} \hat{f}(S)^2$$
\end{thm}

\subsection{Function Restrictions}\label{restrictions subsection}
Let $f:\{0,1\}^{[n]\choose 2}\to \mathbb{R}$  be an arbitrary function and $H\subset {[n]\choose 2}$.  
Given a setting $\beta\in \{0,1\}^{H^c}$ of the edges not in $H$ we define the restricted function $f_\beta:\{0,1\}^H\to \mathbb{R}$ by
$$f_\beta(\alpha):=f(\alpha,\beta)$$
Whenever we use this restriction notation the choice of $H$ will be made explicit beforehand, but not referenced in the notation for the sake of compactness.
For any $S\subset H$ we may express the Fourier coefficients of $\widehat{f_\beta}(S)$ in terms of the coefficients of $f$ as follows:
$$\widehat{f_\beta}(S)=\sum_{T\subset H^c} \hat{f}(S\cup T) \chi_T(\beta)$$

Additionally, sometimes we wish to restrict by looking at some set of vertices, and only considering edges incident to those vertices.  To this end we will define the notion of the vertex support of some set of edges.
\begin{define}
Let $S\subset {[n]\choose 2}$  then $\supp(S)$ is defined to be the set of vertices incident to an edge in $S$.  If we interpret an edge as a set of two vertices, then $\supp(S)=\cup_{e\in S} e$.
\end{define}

\subsection{Characteristic Functions}
The bulk of this paper will be concerned with estimating the characteristic function of $\k$, our normalized copy of $f_r$, the $K_r$ counting random variable.  First we recall the definition of the characteristic function:
\begin{define}
Let $X$ be a random variable.  Then its characteristic function $\varphi_X:\mathbb{R}\to \mathbb{C}$ is defined to be
$$\varphi_X(t):=\E[e^{itX}]$$
\end{define}
These are very well studied objects, and they completely determine their associated random variable.  In particular, we will need the following inversion
formula which specifies the probability distribution of a latice valued random variable in terms of its characteristic function.
\begin{thm}[Fourier Inversion for Lattices] \label{inversion}
Let $X$ be a random variable supported on the lattice $\mathcal{L}:=b+h\mathbb{Z}$.  Let $\varphi_X(t):=\E[e^{itX}]$, the characteristic function of $X$.  Then for $x\in\mathcal{L}$
$$\mathbb{P}(X=x)=\frac{h}{2\pi}\int_{-\frac{\pi}{h}}^{\frac{\pi}{h}} e^{-itx}\varphi_X(t)dt$$
\end{thm}
For a proof, see Theorem 4 of chapter 15.3 in volume 2 of Feller \cite{Feller2}.
We will also need the following bounds on the characteristic function of the Bernoulli random variable.
\begin{lem}\label{bernoulli}
Let $Y$ be a random variable taking the value $1$ with probabillity $p$ and $-1$ with probability $1-p$.  For any $|t|<\frac\pi 2$,    $|\E[e^{itY}]|<1-\frac{8p(1-p) t^2}{\pi^2}$.  Consequently, it also follows that for $|t|<\pi$  we have $|\E[e^{itx_e}]|\le 1-\frac{4p(1-p)t^2}{\pi^2}$ and also for $|t|<\sqrt{p(1-p)}\pi$ we have $|\E[e^{it\chi_e}]|\le 1-\frac{2t^2}{\pi^2}$.
\end{lem}
\begin{proof}
For $Y$ we note that
\begin{align*}
|\E[e^{itY}]|^2&=|pe^{it}+(1-p)e^{-it}|^2=\cos^2(t)+(1-2p)^2\sin^2(t)=1-4p(1-p)\sin^2(t)\\
&\le 1-\frac{16p(1-p)t^2}{\pi^2}
\end{align*}
where the first inequality used the fact that $\sin(t)\ge 2t/\pi$ for $|t|\le \frac{\pi}{2}$.  The first claimed result now follows by noting that
$\sqrt{1-x}\le 1-\frac x 2$.

For the subsequent claim we note that  for any random variable $X$, and $a,b\in \mathbb{R}$
$$\E[e^{it(aX+b)}]=e^{itb}\E[e^{i(at)X}]$$
Since $x_e=(Y+1)/2$ and $\chi_e=\frac{x_e-p}{\sqrt{p(1-p)}}=\frac{Y+1-2p}{2\sqrt{p(1-p)}}$, the result follows.  For example:
$$|\E[e^{it\chi_e}]|=|\E[e^{i\frac{t}{2\sqrt{p(1-p)}}Y}]|\le 1-\frac{8p(1-p)t^2}{4p(1-p)\pi^2}$$
\end{proof}

\subsection{Concentration of low degree polynomials}

Throughout, we will lean heavily on the following hypercontractivity bounds, which roughly says that low degree polynomials in the
$\chi_e$  are reasonably well behaved in terms of moments and concentration.  The reference given here is for Ryan O'Donnell's textbook, but the results do not originate there and are due to a series of results by Bonami, Beckner, Borell and others.  See the notes in \cite{ODonnell} for further reference on the matter.

\begin{thm}[\cite{ODonnell} Theorem 10.24]\label{concentration hypercontractivity}
If $f$ has degree at most $d$ then for any $t\ge (2e/\lambda)^{d/2}$ (recall $\lambda=\min(p,1-p)$),
$$\Pr\left[|f(X)|\ge t\|f\|_2\right]\le \lambda^k\exp\left(-\frac{d}{2e}\lambda t^{2/d}\right)$$
\end{thm}

\begin{thm}[\cite{ODonnell} Theorem 10.21]\label{moment hypercontractivity}
If $f$ has degree at most $d$ then for $q\ge 1$
$$\E[|f|^{2q}]\le (2q-1)^{dq}\lambda^{2d-q}\|f\|_2^{2q}$$
\end{thm}

%

\section{Main Results} \label{mainsection}
In this section we give an overview of the proof of our local limit theorems and statistical distance bounds without the proofs of our lemmas and calculations which will follow in subsequent sections.  In this section we will use the following notation for the density and characteristic functions of 
the standard unit normal $N(0,1)$ respectively
\begin{align*}
\mathcal{N}(x)&:=\frac{1}{\sqrt{2\pi}}e^{-\frac{x^2}{2}}\\
\varphi(t):&=e^{-\frac{t^2}{2}}
\end{align*}

\subsection{Tools For Proving Local Limit Theorems}
Our main engine is the Fourier inversion formula given by Theorem \ref{inversion}.
Using this theorem, we can obtain our local limit theorem and statistical distance bounds.  To this end we cite the following lemmas.  For proofs see
\cite{triangles} (although the ideas do not originate there).
\begin{lem}\label{Pointwise Convergence}
Let $X_n$ be a sequence of random variables
supported in the lattices $\mathcal{L}_n=b_n+h_n\mathbb{Z}$, then
$$|h_n\mathcal{N}(x)-\mathbb{P}(X_n=x)|\le h_n\left(\int_{-\frac{\pi}{h_n}}^{\frac{\pi}{h_n}}\left|\varphi(t)-\varphi_n(t)\right|dt+e^{-\frac{\pi^2}{2h_n^2}}\right)$$
\end{lem}
\begin{lem}\label{L1 Lemma}
Let $X_n$ be a sequence of random variables supported in the lattice 
$\mathcal{L}_n:=b_n+h_n\mathbb{Z}$,
and with chf's $\varphi_n$.  Assume that the following hold:
\begin{enumerate}
\item $\sup_{x\in \mathcal{L}_n} |\Pr(X_n=x)-h_n\N(x)|<\delta_n h_n$
\item $\Pr(|X_n|>A)\le \epsilon_n$
\end{enumerate}
Then $\sum_{x\in \mathcal{L}_n} |\Pr(X_n=x)-\N(x)|\le 2A \delta_n+\epsilon_n+\frac{h_n}{\sqrt{2\pi}A}e^\frac{-A^2}{2}$.
\end{lem}
%

\subsection{Proofs of Main Results}
The main calculation of this paper is the following characteristic function bound:
\begin{thm}\label{MainKr}
Fix $0<\tau<\min(1/2r,~1/12)$.  Recall $\k:=(f_r-\mu)/\sigma$, and let $\varphi_\k(t)$ be the characteristic function of $\k$.  Then 
$$\int_{-\pi \sigma}^{\pi \sigma}\left|\varphi_\k(t)-e^\frac{-t^2}{2}\right|=O(n^{-1/2+2\tau})$$
\end{thm}
\begin{proof}
This proof is a combination of our estimates for the characteristic function $\varphi_\k(t)$ from sections \ref{smallt section} through \ref{larget section} .  The relevant bounds are
\begin{itemize}
\item For $|t|\le n^\tau$, we use Lemma \ref{smallest char bound} to say that $|\varphi_\k(t)-e^{-t^2/2}|=O(n^{-1/2+\tau})$.
\item For $n^\tau<|t|\le n^{\frac12+2\tau}$ we use Lemma \ref{subgraphmain2} to say that $|\varphi_\k(t)|=O(n^{-50})$.
\item For $n^{\frac12+2\tau}<|t|\le n^{\frac{r}{2}-5/12-2\tau}$ Corollary \ref{midchfbound} implies that $|\varphi_\k(t)|=O(n^{-r^2})$
\item For $n^{\frac{r}{2}-5/12-2\tau}<|t|\le n^{r-1-2(r-2)\tau}$ Corollary \ref{midhighchfbound} implies that $|\varphi_\k(t)|=\exp(-\Omega(n^\tau/2r^2))$.
\item For $n^{r-1-2(r-2)\tau}<|t|\le \pi \sigma_n$ Lemma \ref{highchfbound} tells us that $|\varphi_\k(t)|
=\exp(-\Omega(n^{1-2(r-2)\tau})$.
\end{itemize}
Note that in order for the last item on this list to be an effective bound, we require $\tau<\frac{1}{2r-2}$, which is satisfied.
Combining all of these pieces we find that
\begin{align*}
\int_{-\pi \sigma_n}^{\pi \sigma_n}\left|\varphi_\k(t)-e^\frac{-t^2}{2}\right|\le \int_{-n^\tau}^{n^\tau} \frac{t}{\sqrt{n}}dt+2\int_{n^{-\tau}}^{\pi\sigma} e^{-\frac{t^2}{2}}dt+2\int_{n^\tau<|t|\le \pi\sigma}|\varphi_\k(t)|dt=O(n^{-\frac12+2\tau})
\end{align*}
\end{proof}

Theorem \ref{Sup Main} is now just a restatement of the following corollary:
\begin{cor}\label{Sup Bound}
Let $\mathcal{L}_n:=\frac{1}{\sigma}(\mathbb{Z}-\mu)$.  Then for any $x\in \mathcal{L}_n$
$$\left|\mathbb{P}(\k=x)-\frac{\N(x)}{\sigma_n}\right|=O\left(\frac{1}{\sigma n^{\frac12-2\tau}}\right)$$
\end{cor}
\begin{proof}
Apply Lemma \ref{Pointwise Convergence} to $\k$ (where $h_n=1/\sigma$ and $b_n=\mu$), combined with the estimate for the characteristic function of $Z$ given by Theorem \ref{MainKr}.
\end{proof}

Next we prove that $f_r$ and the discrete Gaussian are close in the $\ell^1$ metric as well.
\begin{repthm}{L1 Main}
$$\sum_{m\in \mathbb{N}}\left|\Pr(f_r=m)-\frac{1}{\sqrt{2\pi}\sigma}\exp\left(-\frac{\left(t-\mu\right)^2}{2\sigma^2}\right)\right|=O(n^{-\frac12+2\tau})$$
\end{repthm}
\begin{proof}
It is equivalent to show that for $\mathcal{L}=\frac{1}{\sigma}(\mathbb{Z}-\mu)$ we have
$$\sum_{x\in \mathcal{L}_n}\left|\Pr(\k=x)-\frac{1}{\sigma}\N(x)\right|=O(n^{-1/2+2\tau})$$
This follows from Lemma \ref{L1 Lemma}.  We may set $\delta_n=O(1/\sigma n^{1/2-2\tau})$ by Corollary \ref{Sup Bound}.  We may also set $A=\log(n)^2$ and take $\epsilon_n=\Pr(|\k|\ge \log(n)^2)=O(1/n)$ (this can be shown in several ways.   Lemma \ref{concentration hypercontractivity} will suffice for our purposes but much stronger tools exist).  The main term will then be
$$\sum_{x\in \mathcal{L}_n}\left|\Pr(Z=x)-\frac{1}{\sigma}\N(x)\right|\le O\left(\log^2(n)n^{-1/2+2\tau}+\frac{1}{n}+n^{-\omega(1)}\right)$$
Since choice of $\tau$ was arbitrary, this is sufficient to prove our result.
\end{proof}

\section{Proof of Theorem \ref{mainchf}}\label{mainchf section} 

Let $X=\sum_{i=1}^n a_iX_i$ be a sum of independent $p$-biased mean 0 variance 1 Bernoulli random variables.  Let $Y$ be a degree $d$ polynomial in the $X_i$ such that $Y$
contains no degree 1 monomials.  Assume $\sum_{i=1}^n a_i^2=T$,  and $a_i^2\le \delta$ for all $i$.  Set $\eta:=\|Y\|_2$ and $\epsilon:= \exp[-2t^2(T-\delta d\ell)/\pi^2]$ .  Let $\varphi_X:=\E e^{itx}$ and $\varphi(t)=\E[e^{it(X+Y)}]$ characteristic functions of $X$ and $X+Y$ respectively.
\begin{repthm}{mainchf}
Fix some $\ell\in \mathbb{N}$.   Then for all $t$ such that
 $$|t|<\min\left(\sqrt{p(1-p)}\pi\delta^{-1/2},~~(2e)^\frac \ell 2 \lambda^{-\frac \ell 2}\eta^{-1}\right)$$
 it follows that
\begin{align*}
|\varphi(t)-\varphi_X(t)|&\le  \ell\epsilon\left(1+\left|t\hat{\|}Y\hat{\|}_1\right|^\ell\right)+\frac{|t\eta|^{\ell+1}}{(\ell+1)!} \ell^{\frac{d(\ell+1)}{2}}\lambda^{d\left(\frac{1-\ell}{2}\right)} +\lambda^d\exp\left[-\frac{d\lambda}{2e}\left|t\eta\right|^{-2/d}\right]\\
&\qquad+|t\eta|^{\frac{\ell+1}{2}}\lambda^{\frac{3d-\ell}{4}}(\ell+1)\exp\left[-\frac{d\lambda}{4e}\left|t\eta\right|^{-2/d}\right]
\end{align*}
\end{repthm}

\begin{proof}
\begin{align*}
\varphi(t)-\varphi_X(t)&=\E[e^{it(X+Y)}]-\E[e^{itX}]=\E[e^{it(X+Y)}-e^{itX}]=\E[e^{itX}\left(e^{itY}-1\right)]
\end{align*}
By hypothesis we have that $\|tY\|_2\le (2e)^{\ell/2}\lambda^{-\ell/2}$, so an application of Theorem \ref{concentration hypercontractivity} yields
$$\Pr(|Yt|\ge 1)\le\lambda^d \exp\left(-\frac{d}{2e}\left(t\eta\right)^{-2/d}\right)$$
And when $|Yt|\le 1$ we can use the degree $\ell$ Taylor polynomial of $e^{itY}$ to say that
\begin{align*}
\left|e^{itY}-1-\left(\sum_{j=1}^\ell \frac{t^jY^j}{j!}\right)\right|\le \frac{e|t^{\ell+1}Y^{\ell+1}|}{(\ell+1)!} 
\end{align*}
To acount for the unlikely event that $tY$ is too large to use this Taylor bound let $A$ be the event that $|tY|\ge 1$.  Now let $Z$ be the error random variable $Z:=1_A\left( \left|e^{itY}-\sum^\ell t^jY^j/j!\right|- \frac{e|t^{\ell+1}Y^{\ell+1}|}{(\ell+1)!}\right)$.  
So now we can categorically say that always
$$\left|e^{itY}-1-\left(\sum_{j=1}^\ell \frac{t^jY^j}{j!}\right)\right|\le \frac{e|t^{\ell+1}Y^{\ell+1}|}{(\ell+1)!} +Z$$
We show that $|Z|$ has small expectation.  $|Z|\le 1+(k+1)|tY|^{k+1}$ uniformly.  By Theorem \ref{concentration hypercontractivity}:
$$\Pr(Z\neq 0)=\Pr(A)=\Pr(|tY|\ge 1)\le \lambda^d\exp\left(-\frac{d\lambda}{2e}\left(t\eta\right)^{-2/d}\right)$$

 Because $Y$ is a degree $d$ polynomial with $\|Y\|_2=\eta$ by Theorem \ref{moment hypercontractivity} 
\begin{equation}\label{thm3eq1}
\E|Y|^{\ell+1}=\|Y\|_{\ell+1}^{\ell+1}\le \ell^{\frac{d(\ell+1)}{2}}\lambda^{d\left(\frac{1-\ell}{2}\right)} \eta^{\ell+1}
\end{equation}
So it follows by Cauchy-Schwarz that
 \begin{align}\label{thm3eq2}
 \E|Z|&\le \E\left[1_A\cdot (1+(\ell+1)|tY|^{\ell+1})\right]\le \Pr(A)+\|1_A\|_2\|(\ell+1)|tY|^{\ell+1}\|_2\nonumber\\
 &\le\lambda^d\exp\left[-\frac{d\lambda}{2e}\left(t\eta\right)^{-2/d}\right]+\left(\lambda^d\exp\left[-\frac{d\lambda}{2e}\left(t\eta\right)^{-2/d}\right]\right)^\frac12\left((\ell+1)^2t^{\ell+1}  \ell^{\frac{d(\ell+1)}{2}}\lambda^{d\left(\frac{1-\ell}{2}\right)} \eta^{\ell+1}\right)^\frac12\nonumber\\
 &=\lambda^d\exp\left[-\frac{d\lambda}{2e}\left(t\eta\right)^{-2/d}\right]+(t\eta)^{\frac{\ell+1}{2}}\lambda^{\frac{3d-\ell}{4}}(\ell+1)\exp\left[-\frac{d\lambda}{4e}\left(t\eta\right)^{-2/d}\right]
 \end{align}

Next, we analyze the Taylor polynomial of $e^{itY}$ by splitting $Y^j$ into a sum of monomials.  Let $\M_j$ be the set of monomials supported in $Y^j$ and say
$$Y^j=\sum_{m\in \M_j} a_mm$$
as a result we may write
$$\E\left[e^{itX}\frac{t^jY^j}{j!}\right]=\frac{t^j}{j!}\sum_{m\in \M_j} a_m\E\left[e^{itX}m\right]$$

We now examine this one monomial at a time.  Let $M$ denote the set of variables $X_i$ appearing in some fixed monomial $m$.  Note that because $Y$ has degree $d$ and $m$ is in $Y^j$ we have that $|M|\le dj\le d\ell$.  Then:

\begin{align}\label{thm3eq3}
\E[e^{itX}m]=\E_{M}\E_{M^c}[me^{itX}]=\E_{M}me^{it\sum_{i\in M} a_iX_i}\E_{M^c}[e^{it\sum_{i\in M^c} a_iX_i}]
\end{align}
So we have that
$$\left|\E[e^{itX}m]\right|=\left|\E_{M}me^{it\sum_{i\in M} a_iX_i}\E_{M^c}[e^{it\sum_{i\in M^c} a_iX_i}]\right|
\le |m|\left|\E_{M^c}[e^{it\sum_{i\in M^c} a_iX_i}]\right|$$

Using the hypothesis $|a_it|<\sqrt{p(1-p)}\pi$, we obtain from Lemma \ref{bernoulli} that
\begin{align*}
\left|\E_{M^c}[e^{it\sum_{i\in M^c} a_iX_i}]\right|=&\prod_{i\in M^c}\left|\E[e^{ita_iX_i}]\right|\le \prod_{i\in M^c}\left(1-\frac{2}{\pi^2}(t^2a_i^2)\right)
\le e^{-\frac{2}{\pi^2}t^2\sum_{i\in M^c}a_i^2}\\
&\le  e^{-\frac{2}{\pi^2}t^2(T-\delta|M|)}\le e^{-\frac{2}{\pi^2}t^2(T-\delta d\ell)}\\
&\le \epsilon
\end{align*}

So plugging this back into equation \ref{thm3eq3} we find that 
$$\left|\E\left[e^{itX}\frac{t^jY^j}{j!}\right]\right|\le \sum_{m\in \M_j}\frac{|t|^j}{j!}\left|\E\left[e^{itX}m\right]\right|\le\frac{|t|^j}{j!}\sum_{m\in \M_j} |a_m m|\epsilon\le  \frac{|t|^j}{j!}\epsilon C_p^{dj} \hat\|Y^j\hat\|_1\le  \frac{|tC_p^d|^j}{j!}\epsilon \hat\|Y\hat\|_1^j$$
Where the last inequality uses Lemma \ref{OneNorm} and the constant
$$C_p=\left(1+\frac{|1-2p|}{\sqrt{p(1-p)}}\right)\|\chi_p\|_\infty=\left(1+\frac{|1-2p|}{\sqrt{p(1-p)}}\right)\sqrt{\frac{1-\lambda}{\lambda}}$$

Summing over all $j$ from 1 to $\ell$
\begin{align*}
|\varphi(t)-\varphi_X(t)|&=\left|\E\left[e^{itX}(e^{itY}-1)\right]\right|\le  \left|\sum_{j=1}^\ell\E\left[e^{itX}\frac{t^jY^j}{j!}\right]\right|
+\E\left|e^{itX}\left( \frac{e|t^{\ell+1}Y^{\ell+1}|}{(\ell+1)!} +Z\right)\right|\\
&\le \left(\sum_{j=1}^\ell  \frac{C_p^{dj}|t|^j}{j!}\epsilon \hat\|Y\|_1^j\right)+\frac{|t|^{\ell+1}}{(\ell+1)!}\E|Y|^{\ell+1} +\E|Z|\\
&\le \ell\epsilon\left(1+\left(tC_p^d\hat{\|}Y\hat{\|}_1\right)^\ell\right)+\frac{|t|^{\ell+1}}{(\ell+1)!}\E|Y|^{\ell+1} +\E|Z|
\end{align*}
Where the $1+\ldots$ inside the first parenthesis on the last line is to account for the possibility that $t\hat\|Y\hat\|_1\le 1$.  We may use equations
\ref{thm3eq1} and \ref{thm3eq2} to bound the two error terms in the above equation.    Putting all of the estimates together yields

\begin{align*}
|\varphi(t)-\varphi_X(t)|&=\left|\E\left[e^{itX}(e^{itY}-1)\right]\right|\le  \left|\sum_{j=1}^\ell\E\left[e^{itX}\frac{t^jY^j}{j!}\right]\right|
+\E\left|e^{itX}\left( \frac{e|t^{\ell+1}Y^{\ell+1}|}{(\ell+1)!} +Z\right)\right|\\
&\le \left(\sum_{j=1}^\ell  \frac{|t|^j}{j!}\epsilon \|Y\|_1^j\right)+\frac{|t|^{\ell+1}}{(\ell+1)!}\E|Y|^{\ell+1} +\E|Z|\\
&\le \ell\epsilon\left(1+\left|t\hat{\|}Y\hat{\|}_1\right|^\ell\right)+\frac{|t|^{\ell+1}}{(\ell+1)!}\E|Y|^{\ell+1} +\E|Z|\\
&\le  \ell\epsilon\left(1+\left|t\hat{\|}Y\hat{\|}_1\right|^\ell\right)+\frac{|t\eta|^{\ell+1}}{(\ell+1)!} \ell^{\frac{d(\ell+1)}{2}}\lambda^{d\left(\frac{1-\ell}{2}\right)} +\lambda^d\exp\left[-\frac{d\lambda}{2e}\left|t\eta\right|^{-2/d}\right]\\
&\qquad+|t\eta|^{\frac{\ell+1}{2}}\lambda^{\frac{3d-\ell}{4}}(\ell+1)\exp\left[-\frac{d\lambda}{4e}\left|t\eta\right|^{-2/d}\right]
\end{align*}
\end{proof}

\begin{lem}\label{OneNorm}
Let $f$ be a polynomial.  For any $j\in \mathbb{N}$ we have that
$$\|f^{j}\|_1\le \left[1+\frac{|1-2p|}{\sqrt{p(1-p)}}\right]^{j-1}\hat\|f\hat\|_1^{j}$$
\end{lem}
\begin{proof}
We prove this by induction on $j$.  For $j=1$ the statement is trivial.  Now assume it is true for some arbitrary $j$.  Let $\M$ be the set of monomials supported in $f$ and $f=\sum_\M a_mm$.  Then
\begin{align*}
\hat\|f^{j+1}\hat\|_1&=\hat\|\sum_{m\in \M} a_mmf^j\hat\|_1\le \sum_{m\in \M} |a_m|\hat\|mf^j\hat\|_1\le \sum_{m\in \M} |a_m|\left(1+\frac{|1-2p|}{\sqrt{p(1-p)}}\right)\hat\|f\hat\|_1^j\\
&\le \left[1+\frac{|1-2p|}{\sqrt{p(1-p)}}\right]^j\hat\|f\hat\|_1^{j+1}
\end{align*}
completing the induction and the proof.
\end{proof}

%

\section{Decoupling and Polynomial Degree Reduction}\label{decoupling section}

In this section we set up our decoupling technique for reducing the degree of polynomials in characteristic function computations.  For an example
illustrating the idea, see Example \ref{cs example} in the introduction.

%
%
%
\subsection{The $\alpha$ operator}\label{alpha def section}
Partition the edge set of ${[n]\choose 2}$ into $k+1$ parts $B_0,B_1,\ldots,B_k$.  Let $X\in\{0,1\}^{B_0}$ be the indicator random variables denoting which edges from $B_0$ are in our graph sampled from $G(n,p)$.
  Similarly for $i\in [k]$ let $Y_i$ denote the indicator random variables for the edges in $B_i$.  Also, for any $i\in [k]$ let $Y_i^{0}$  and $Y_i^1$ denote independent random variables with the same marginals as $Y_i$.
  We let $B_i^0$ and $B_i^1$ denote separate copies of the edges in $B_i$, so that we may say $Y_i^j\in \{0,1\}^{B_i^j}$.
    For $v\in \{0,1\}^k$ let $Y^v:=(Y_1^{v_1},Y_2^{v_2},\ldots,Y_k^{v_k})$.
    
      One may interpret this as choosing two random samples $Y_i^0$ and $Y_i^1$ from $G(n,p)$ for the edges in each $B_i$.  Then $X,Y^v$ would correspond to a sampling of every edge in the graph, with which of the copies of each $Y_i$ you query controlled by 
  the binary vector $v$.  
  
Finally, set  $\B:=\cup_{i=1}^k\cup_{j=0}^1 B_i^j$, and let  $\Y:=(Y_1^0,Y_1^1,Y_2^1,\ldots, Y_k^1)\in \{0,1\}^\B$.  For $v\in \{0,1\}^k$ let $|v|$ denote the hamming weight of $v$.

  We are now ready to define our $\alpha$ operator.
  
\begin{define}\label{alpha}
Given the partition ${[n]\choose 2}=B_0\cup B_1\cup\ldots\cup B_k$ as above, define the operator on  functions of the form $f(X,Y_1,\ldots,Y_k)$, which outputs 
the function $\alpha(f):\{0,1\}^{B_0\cup\B}\to \mathbb{R}$ given by
\begin{align*}
\alpha(f)(X,\Y)&:=\sum_{\textbf{v}\in \{0,1\}^k}(-1)^{|v|}f(X,Y^v)
\end{align*}
\end{define}

Note that $\alpha$ is a linear operator, and in the rest of this section we will describe its action.  First we define a pair of terms which will be useful in our analysis.
\begin{define}[Rainbow Sets]
 We call a set $S\subset {[n]\choose 2}$ \textbf{rainbow} if $S\cap B_i\neq \varnothing$ for all $1\le i \le k$.
\end{define}
\begin{define}[Flattening from $\B$ to ${[n]\choose 2}$]
Given a set $S\subset B_0\cup \B$ define $$\fl(S):=\{e\in {[n]\choose 2}\st e^0\mbox{ or }e^1 \in S\}$$  That is, $\fl(S)$ takes in a subset of $\B$, and outputs
the set of edges in ${[n]\choose 2}$ relevant to $S$ (information about which copy or both of $e$ were in $S$ is omitted).
\end{define}

  \subsection{Action of $\alpha$ on $\chi_S$}
In this subsection we compute the action of $\alpha$ on our basis functions $\chi_S$.  Write $S=S_0\cup S_1\cup\ldots\cup S_k$ where
$S_i=S\cap B_i$.  Because $\chi_S:=\prod_{i=0}^k \chi_{S_i}$ we can compute that
\begin{align*}
\alpha(S)&:=\alpha(\chi_S)=\sum_{\textbf{v}\in \{0,1\}^k}(-1)^{|v|}\chi_S(X,Y^v)=\sum_{\textbf{v}\in \{0,1\}^k}\chi_{S_0}(X)\prod_{i=1}^k(-1)^{v_i}\chi_{S_i}(Y_i^{v_i})\\
&=\chi_{S_0}(X)\prod_{i=1}^k\left(\chi_{S_i}(Y_i^0)-\chi_{S_i}(Y_i^1)\right)
\end{align*}
If $S$ is not rainbow, then for some $i$ we have $S_i=\varnothing$, so it follows from the above product form that if $\alpha(S)\equiv 0$.  Furthermore, if $S$ is rainbow, then

\begin{align*}
\E \alpha(S)&=\hat\alpha(\varnothing)=0\\
\|\alpha(S)\|_2^2&=\sum_{T\subset \Y}\hat\alpha(T)^2=\sum_{v\in \{0,1\}^k} (-1)^{2|v|}=2^k
\end{align*}
We can also compute that for $S, T$, both rainbow we have
\begin{align*}
\E[\alpha(\chi_S)\alpha(\chi_t)]&=\E\left[\chi_{S_0}\chi_{T_0}\right]\prod_{i=0}^k\E\left[\left(\chi_{S_i}(Y_i^0)-\chi_{S_i}(Y_i^1)\right)\left(\chi_{T_i}(Y_i^0)-\chi_{T_i}(Y_i^1)\right)\right]\\
&=\begin{cases}
2^k\mbox&\mbox{if }S_i=T_i\neq\varnothing~\forall i\\
0&\mbox{else}
\end{cases}
\end{align*}
So we have shown that the linear operator $\alpha$ is orthogonal for $S$ rainbow, and 0 on $S$ nonrainbow.  In particular, if
$f=\sum_{S} \hat{f}(S)\chi_S$ then
\begin{align*}
\|\alpha(f)\|_2^2&=\E\left[\left(\sum_{S\subset E} \alpha(\chi_S)\hat{f}(S)\right)^2\right]=\sum_{S\mbox{ rainbow}} 2^k \hat{f}(S)^2
\end{align*}

We will also need to examine products of the form $\alpha(\chi_S)\alpha(\chi_T)$ for distinct sets $S,T\subset {[n]\choose 2}$.

\begin{lem} \label{alpha product}
 Let $S,T\subset {[n]\choose 2}$ and set $\gamma=\frac{1-2p}{\sqrt{p(1-p)}}$.  
For $U\subset B_0\cup \B$ we have
\begin{align*}
\left|\widehat{\alpha(\chi_{S})\alpha(\chi_T)}(U)\right|\le \begin{cases}
0&\mbox{if }S\Delta T\not \subset \fl(U)\mbox{ or } \fl(U)\not \subset S\cup T\\
\max\left(\gamma^{|U|}, 1\right)&\mbox{if } S\Delta T\subset \fl(U)\subset S\cup T
\end{cases}
\end{align*}
\end{lem}
The proof is mostly a calculation and is contained in appendix \ref{alpha product appendix}.

\subsection{$\alpha$ and decoupling}
%
%
%
%
%
%
%
We are now ready to state our main Lemma of this section.

\begin{lem}\label{decoupling}
Let $f:=f(X,Y)$ be a function of the independent random variables $X,Y_1,\ldots,Y_k$.  Let $\varphi(t)=\E[e^{itf(X,Y)}]$ the characteristic function of $f$.  Then we have

$$|\varphi(t)|^{2^k}\le \E_{\Y}\left| \E_X e^{it\alpha(f)(X,\Y)}\right|$$
\end{lem}

\begin{proof}
We prove this statement inductively on $k$.  For $k=0$ the proof is trivial, and for $k=1$ the proof is in Example \ref{cs example}.  

Let $\tilde X=(X,Y_{k})$  and $\tilde Y=Y_1,\ldots,Y_{k-1}$.  Then by the inductive hypothesis:
$$|\varphi(t)|^{2^{k-1}}\le \E_{\tilde Y}\left| \E_{\tilde X} e^{it\sum_{v\in \{0,1\}^{k-1}}(-1)^{|v|} f(\tilde X,\tilde Y^{v})}\right| $$

Applying Cauchy-Schwarz to the inner term of the above expectation we find
\begin{align*}
\big|\E_{X,Y_{k}} &e^{it\sum_{\textbf{v}\in \{0,1\}^{k-1}}(-1)^{|v|} f(X,Y_k,\tilde Y^{v})}\big|^2
\le\E_X\left|\E_{Y_k}e^{it\sum_{v\in \{0,1\}^{k-1}}(-1)^{|v|} f(X,Y_k,\tilde Y^{v})}\right|^2\\
&=\E_X\left(\E_{Y_k^0}e^{it\sum_{v\in \{0,1\}^{k-1}}(-1)^{|v|} f(X,Y_k,\tilde{Y}^{v})}\right)\overline{\left(\E_{Y_k^1}e^{it\sum_{{v}\in \{0,1\}^{k-1}}(-1)^{|v|} f(X,Y_k,\tilde{Y}^{v})}\right)}\\
&=\E_{X,Y_k^0,Y_k^1}e^{it\sum_{v\in \{0,1\}^k} f(X,Y^v)(-1)^{|v|}}=\E_{X,Y_k^0,Y_k^1}e^{it\alpha(f)}
\end{align*}
So a second application of Cauchy-Schwarz now tells us that
\begin{align*}
|\varphi(t)|^{2^k}&\le \left(\E_{\tilde Y}\left| \E_{\tilde X} e^{it\sum_{v\in \{0,1\}^{k-1}}(-1)^{|v|} f(\tilde X,\tilde Y^{v})}\right|\right)^2
\le \E_{\tilde Y}\left| \E_{\tilde X} e^{it\sum_{v\in \{0,1\}^{k-1}}(-1)^{|v|} f(\tilde X,\tilde Y^{v})}\right|^2\\
&\le\E_{\tilde Y}\E_{X,Y_k^0,Y_k^1}e^{it\alpha(f)}\le \E_{\Y}\left|\E_{X}e^{it\alpha(f)(X,\Y)}\right|
\end{align*}

\end{proof}

\section{Properties of the $K_r$ counting function}\label{KrProp}
In this section we compute the Fourier transform of $f_r$, the $K_r$ counting function, and its normalized brother $\k$.  Recall the definition of $f_r$ by
$$f_r=\sum_{\substack{H\subset G\\H\equiv K_r}} 1_H$$
Where the sum is over all copies of $K_r$ in $G$.
 Meanwhile for each individual $r$-clique $H$, its indicator function is given by
$$1_H(X)=\prod_{e\in H} x_e=\prod_{e\in H} \left(\sqrt{p(1-p)}\chi_e+p\right)=p^{r\choose 2}\sum_{\supp(S)\subset H} \left(\frac{1-p}{p}\right)^{|S|/2}\chi_S$$
Summing over all posible choices of $H$ we find that $\widehat{f_r}(S)$ is $p^{r\choose 2}(\frac{1-p}{p})^{|S|/2}$ multiplied by the number of different $r$-cliques containing all of the edges in $S$.  But we know any set $S$ supported on $t$ vertices appears in exactly ${n-t\choose r-t}$
$r$-cliques.  Therefore
\begin{equation}
f_r=p^{r\choose2}\sum_{t=0}^r{n-t\choose r-t}\left(\frac{1-p}{p}\right)^{|S|/2}\sum_{|\supp(S)|=t}  \chi_S
\end{equation}
Combining this formula with Theorem \ref{parseval} allows us to quickly compute $\sigma^2$, the Variance $f_r(G)$ when $G$ is drawn from $G(n,p)$.
\begin{align}
\sigma^2:=Var_{G\sim G(n,p)}(f(G))&=\sum_{S\neq \varnothing} \hat{f}(S)^2=p^{r(r-1)}\sum_{t=2}^r\left[{n-t\choose r-t}\left(\frac{1-p}{p}\right)^{|S|/2}\right]^2\sum_{|\supp(S)|=t}  1\\
&=\frac{p^{r(r-1)-1}(1-p)}{2(r-2)!^2} n^{2r-2}+O\left(n^{2r-3}\right)\nonumber
\end{align}

As a reminder recall that $\k=\frac{f_r-\mu}{\sigma}$ is the normalized (mean 0, variance 1) rescaling of $f_r$.  We note that
\begin{align}
W^{1}(\k)&=\sum_{|S|=1} \hat{\k}(S)^2=1-O\left(\frac1n\right)\\
W^{>1}(\k)&=\sum_{|S|\ge 2} \hat{\k}(S)^2=\Theta\left(\frac1n\right)\\
\hat{\k}(S)&=\frac{p^{r\choose 2}\left(\frac{1-p}{p}\right)^{|S|/2}{n-|\supp(S)|\choose r-|\supp(S)|}}{\sigma}=\Theta\left(n^{1-|\supp(S)|}\right)
\end{align}
Where the above formula for $\hat{\k}(S)$ is valid for all $S\neq \varnothing$ (and $\hat{\k}(\varnothing)=0$).

\section{Bound for $|t|\le n^{\tau}$}\label{smallt section}
In this section we prove the following Lemma:
\begin{lem}\label{smallest char bound}
For all $t=o(n)$ we have $|\varphi_\k(t)-e^{-t^2/2}|=O\left(\frac{t}{\sqrt{n}}\right)$.
\end{lem}
  In order to do this, we will need the Berry-Esseen theorem.  The following lemma is a restatement of Lemma 1 of Chapter V in Petrov's Sums of Independent Random Variables \cite{Petrov}.
\begin{lem}\label{BELemma}
Let $Q^2=\frac{1}{{n\choose 2}}$ and set $X=\sum_{e\in {[n]\choose 2}} Q\chi_e$.  It is the mean 0 variance 1 sum of independent  random variables.  Further define $L_n$ to be
$$L_n:={n\choose 2}\E[|Q\chi_e|^3]=\frac{p^2+(1-p)^2}{\sqrt{{n\choose 2}p(1-p)}}=\Theta_p\left(1/n\right)$$
then for $t\le \frac{1}{4L_n}$ we have that
\begin{equation}\label{eqBerryEsseen}
\left|\E[e^{itX}]-e^{-\frac{t^2}{2}}\right| \le 16L_n|t|^3e^{\frac{-t^2}{3}}
\end{equation}
\end{lem}
With this bound, we are ready to prove Lemma \ref{smallest char bound}. 
\begin{proof}[Proof of Lemma \ref{smallest char bound}]
Decompose $\k$ into two parts:  $X$ a mean 0 variance 1 sum of i.i.d.\ random variables, and $Y$, which is considered as an error term.   Set $Q=\frac{1}{\sqrt{n\choose 2}}$ and let
\begin{align*}
X:=\sum_{e\in {[n]\choose 2}} Q \chi_e&&Y:=\k-X=\sum_{e\in {[n]\choose 2}} (\hat \k(e)-Q)\chi_e+\sum_{|S|\ge 2} \hat \k(S) \chi_S
\end{align*}
We know that all edges $e\in {[n]\choose 2}$ have the same Fourier coefficient $\hat{\k}(e)$, and further that 
$$\sum_{e}\hat{\k}(e)^2={n\choose 2}\hat{\k}(e)^2=1-W^{>1}(\k)=1-O\left(\frac1n\right)$$
Therefore it follows that 
$$|\hat{\k}(e)-Q|=\left| \frac{\hat{\k}(e)^2-Q^2}{\hat{\k}(e)+Q}\right|=O\left(\frac1{n^2}\right)$$
So now we can compute
$$\|Y\|_2^2=\sum_{e} (\hat{\k}(e)-Q)^2+\sum_{|S|\ge 2} \hat{\k}(S)^2=O\left(\frac1{n^2}\right)+W^{>1}(\k)=O\left(\frac1n\right)$$

For $t\le \frac{1}{4L_n}=\Theta(n)$, Lemma \ref{BELemma}, the above calculation, and Cauchy-Schwarz applied to $\E[|Y|]^2$ tell us

\begin{align*}
\left|\varphi_\k(t)-e^{-\frac{t^2}{2}}\right|&=\left|\E\left[e^{it\k}\right]-e^{-\frac{t^2}{2}}\right|=\left|\E\left[e^{it(X+Y)}\right]-e^{-\frac{t^2}{2}}\right|\le \left|\E\left[e^{itX}\right]-e^{-\frac{t^2}{2}}\right|+\left|\E\left[e^{itX+Y}\right]-\E e^{itX}\right|\\
&\le 16L_n|t|^3e^{\frac{-t^2}{3}}+\E|tY|=O\left(\frac{t^3e^{-\frac{t^2}{3}}}{n}+\frac{t}{\sqrt{n}}\right)
\end{align*}
\end{proof}

\section{Bound for $|t|\in [n^{\tau}, n^{\frac12+2\tau}]$} \label{tier2 section}
The goal for this section is to prove the following lemma
\begin{lem}\label{subgraphmain2}
For $n^{\tau}<t\le n^{\frac34}$
$$|\varphi_\k(t)|\le O\left(n^{-50}\right)$$
\end{lem}
\subsection{High Level Proof}
In this subsection, we will assume the following helper claims, and then prove Lemma \ref{subgraphmain2}.

\begin{claim}\label{goodgraph}
For all sufficiently large $n$ and any $\alpha\in (n^{-1+\tau},1)$
there exists a set of edges $H\subset {[n] \choose 2}$  with $|H|\ge \alpha{n\choose 2}$ such that
$$\sum_{\substack{S\subset H\\|S|\ge 2}} n^{2-2|\supp(S)|}=O( \alpha^2 n^{-1})$$
\end{claim}

For the subsequent claims, assume we have chosen one such $H$ as promised by Lemma \ref{goodgraph} which will be fixed throughout.

\begin{claim}\label{GoodEvent}
Let $A$ be the event (over the space of revelations $\beta \in \{0,1\}^{H^c}$) that for \emph{every} edge $e\in H$
$$|\widehat{\k_\beta}(e)-\hat{\k}(e)|<\frac{1}{n^{1.4}}$$
Recall $\lambda:=\min(p,1-p)$.  Then $\Pr(A)\ge 1-n^2\exp\left[-\Omega\left( n^{\frac{0.4}{r^2}}\right)\right]$.
\end{claim}

\begin{claim}\label{goodtimes}
Let $B$  be the event (over the space of revelations $\beta \in \{0,1\}^{H^c}$) that for \emph{every} set $S\subset {[n]\choose 2}$ with $|S|\ge 2$
$$|\widehat{\k_\beta}(S)|=Cn^{r-|S|}$$
where $C$ is a fixed constant depending on $r$ and $p$, but not on $n$.  Then $\Pr(B)\ge 1-\exp\left(-\Omega\left[n^{-1/r^2}\right]\right)$
\end{claim}

\begin{claim}\label {goodcase}
Assume $\beta \in A\cap B$.  Then for $t\in [n^\tau,n^{3/4}]$
$$\E_{H}[ e^{it\k_\beta}]=O(n^{-50})$$
\end{claim}
%

Lemma \ref{subgraphmain2} now follows by combining all of these claims.
\begin{proof}[Proof of Lemma \ref{subgraphmain2}]
Let $A$, and $B$ be as defined in Claims \ref{GoodEvent} and \ref{goodtimes}.  
We can break up $\{0,1\}^{H^c}$ into $A\cap B$ and $(A\cap B)^c$ and estimate

\begin{align*}
|\varphi_\k(t)|&:=\big{|}\E_{(\alpha,\beta)\in 2^{n\choose 2}}[e^{it\k(\alpha,\beta)}]\big|\le \E_{\beta\subset H^c}|\E_{\alpha\subset H}[e^{it\k_\beta(\alpha)}]|\le \Pr[(A\cap B)^c]+ \Pr[A\cap B]\E_{\beta \in (A\cap B)}\left|\E_{\alpha}[e^{it\k_\beta}]\right|
\end{align*}
Combining Claims \ref{GoodEvent} \ref{goodtimes}, and \ref{goodcase} we can bound the right hand side of the above by $O(n^{-50})$.
\end{proof}

\subsubsection{Proof of Claims}
\begin{proof}[Proof Of Claim \ref{goodgraph}]
Draw a random graph $H$ on $n$ vertices by choosing $\alpha{n\choose 2}$ edges uniformly at random.  Then we note that 
\begin{align*}
\E\sum_{\substack{S\subset H\\|S|\ge 2}} n^{2-2|\supp(S)|}=\sum_{\substack{S\subset{n\choose 2}\\2\le |S|\le r}} n^{2-2|\supp(S)|}\Pr(S\subset H)
\le \alpha^2n^2\sum_{i=3}^{r}n^{-2i}\sum_{|\supp(S)|=i} 1=O\left(\alpha^2 n^{-1}\right)
\end{align*}
So some $H$ must have at most the average value for this sum.
%
%
%
\end{proof}

We prove Claim \ref{GoodEvent} by noting that the formula for $\widehat{\k_\beta}(S)$ (a coefficient in the polynomial $\k_\beta$) is \emph{itself} a low degree polynomial, and therefore may be shown to have tight concentration by hypercontractivity. 

\begin{proof}[Proof Of Claim \ref{GoodEvent}]

Recall that from Section \ref{restrictions subsection} that
$$\widehat{\k_\beta}(e)=\sum_{T\subset H^c} \hat \k(e\cup T)\chi_{T}(\beta)$$

So $\widehat{\k_\beta}(e):\{0,1\}^{H^c}\to\mathbb{R}$ is a polynomial (in the functions $\chi_e$), and we can began by estimating its coefficients.
First we see that 
$$\E[\widehat{\k_\beta}(e)] =\widehat{\widehat{\k}_\beta(e)}(\varnothing)=\hat{\k}(e)$$

Also for any $T\subset \{0,1\}^{H^c}$ we know that $\hat{\k}(e\cup T)\neq 0$ only if $|\supp(e\cup T)|\le r$.  So we can compute:

\begin{align*}
Var_\beta( \widehat{\k_\beta}(e))&=\sum_{\substack{T\subset H^c\\T\neq \varnothing}} \hat{\k}(e\cup T)^2=\sum_{i=3}^r \sum_{\substack{T\subset H^c\\|\supp(T\cup e)|=i}} \hat \k(e\cup T)^2\\
&\le \sum_{i=3}^r\sum_{|\supp(T\cup e)|=i} \hat \k(e\cup T)^2\le \sum_{i=3}^r {n-2 \choose i-2}O(n^{2-2i})=O\left(\frac{1}{n^3}\right)
\end{align*}

Since $\widehat{\k_\beta}(e)$ has degree less than ${r\choose 2}$, an application of Theorem \ref{concentration hypercontractivity} gives us that for any $e\in H$ 
$$\Pr\left[\left|\widehat{\k_\beta}(e)-\hat \k(e)\right|\ge \frac{1}{n^{1.4}}\right]<\exp\left(-\Omega\left( n^{\frac{0.4}{r^2}}\right)\right)$$
Applying a union bound over all edges in $H$ completes the proof.
\end{proof}

\begin{proof}[Proof of Claim \ref{goodtimes}]

Again we use the decomposition 
$$\widehat{\k_\beta}(S)=\sum_{T\subset H^c} \hat \k(S\cup T)\chi_{T}(\beta)$$

and note that 
$$\E[\widehat{\k_\beta}(S)] =\widehat{\widehat{\k}_\beta(S)}(\varnothing)=\hat{\k}(S)$$

Let $|\supp(S)|=s$.  For any $T\subset \{0,1\}^{H^c}$ we know that $\hat{\k}(S\cup T)\neq 0$ if and only if $|\supp(S\cup T)|\le r$.  There are at most ${n-s\choose \ell-s}2^{{\ell\choose 2}}\le 2^{r^2}n^{\ell-s}$ choices of $T$ such that $s=|\supp(S\cup T)|=\ell$.  And further for each of these choices we know that $\hat \k(S\cup T)= \Theta(n^{1-\ell})$.  
Define the helper function
$$g:=\sum_{\substack{T\subset H^c\\|\supp(S\cup T)|>s}} \hat \k(S\cup T)\chi_T(\beta)$$
We can compute that
\begin{align*}
Var(g)&\le \sum_{\ell=s+1}^r \sum_{|\supp(S\cup T)=\ell} \left(\hat{\k}(S\cup T)\right)^2 \le \sum_{\ell=s+1}^{r} 2^{r^2}n^{\ell-s} \Theta(n^{2-2\ell})\\
&\le O(n^{2-2s-1})
\end{align*}
Further we can see that $g$ is a polynomial of degree at most ${r\choose2}$, and so by Theorem \ref{concentration hypercontractivity}
\begin{align*}
\Pr\left[|g|\ge n^{1-s+\frac{1}{4}}\right]
=\exp\left[-\Omega\left(n^{1/r^2}\right)\right]
\end{align*}

If $|g|<n^{1-s+1/4}$ then we can conclude that 
\begin{align*}
\hat \k(S)=\sum_{|\supp(S\cup T)|=s} \hat \k(S\cup T)\chi_T(\beta)+g(\beta)\le 2^{{s\choose 2}} \Theta(n^{1-s})+n^{1-s+\frac14}=O(n^{1-s})
\end{align*}
So for any $S\subset H$ we find that 
$|\widehat{\k_\beta}(S)|\le O(n^{1-s})$  with probability at least $1-\exp(-\Omega(n^{\frac{2}{r^2}}))$.  Taking a union bound over all such $S$ finishes the proof.
\end{proof}

\begin{proof}[Proof of Claim \ref{goodcase}]
Fix $\alpha=\frac{1}{t}$ and assume that $\beta\in A\cap B$.  Let $X$ and $Y$ be
\begin{align*}
X:=\sum_{e\in {[n]\choose 2}} \widehat{\k_\beta}(e) \chi_e&&Y:=\sum_{|S|\ge 2} \widehat{\k_\beta}(S) \chi_S
\end{align*}
then $\k_\beta=X+Y$, where $X$ is an independent sum, and $Y$ is small.  To apply Theorem \ref{mainchf} we set $\ell=99$ and compute the relevant parameters to be
\begin{enumerate}
\item $T:= \sum_{e\in H} \widehat{\k_\beta}(e)^2\ge  \sum_{e\in H} \hat{\k}(e)^2/4\ge \alpha{n\choose 2} \frac{1}{4n^2}\ge \frac{1}{t}$ where the last inequality uses the fact that $n^\tau\le t\le n^{\frac34}$.
\item $\displaystyle \eta^2=\|Y\|_2^2= \sum_{\substack{S\subset H\\|S|\ge 2}}\widehat{\k_\beta)}(S)^2\le \sum_{\substack{S\subset H\\|S|\ge 2}}O(n^{2-2|S|})=O(\alpha^2n^{-1})=O\left(\frac{1}{nt^2}\right)$
\item $\hat{\|}Y\hat{\|}_1=\sum_{S\subset H} |\widehat{\k_\beta}(S)|=O(n)$.
\item $\delta=\max_{e}(\widehat{\k_\beta}(e))=3\widehat{\k}(e)/2\le \frac{3}{n}$
\item \begin{align*}
\epsilon&=\exp\left(-\frac{t^2[T-\delta {r\choose 2}\ell]}{\pi^2}\right)=\exp\left(-\Omega\left(t^2\left[t^{-1}-\frac{\ell{r\choose 2}}{n}\right]\right)\right)=\exp\left(-\Omega\left(t\right)\right)=\exp\left(-\Omega\left(n^{\tau}\right)\right)
\end{align*}
\end{enumerate}

Where the third step above used the fact that $t=o(n)$, and the last used that $t\ge n^\tau$.  Given these settings of parameters we can plug into Theorem \ref{mainchf} and find that
$$\E[e^{it\k_\beta}]\le O\left(\epsilon n^\ell+  (n^{-{\frac12}})^{\ell+1}+ \exp\left(-\Omega(n^{1/r^2})\right)+n^{\frac{\ell+2}{4}}  \exp\left(-\Omega(n^{1/r^2})\right) \right)=O(n^{-50})$$
\end{proof}

\section{Bound for $|t|\in [n^{\frac12+2\tau},n^{\frac{r}{2}-\frac{5}{12}-\tau}]$}
\subsection{High Level Overview}
In this section we discuss how to bound the characteristic function  $\varphi_\k(t)$ for $t\in [n^{\frac12+2\tau},n^{\frac{r}{2}-\frac{5}{12}-\tau}]$.  The central trick is the decoupling
tool from Section \ref{decoupling section} combined with Theorem \ref{mainchf}.  The basic outline is that we will partition the edges of ${[n]\choose 2}$ into $k+1$ pieces, apply Lemma \ref{decoupling} to switch our attention from $\k$ to $\alpha(\k)$, and then further examine a random restriction
to edges on some subset $U_0\subset [n]$ of the vertices.  The restricted polynomial $\alpha(\k)_\Y$ will have its Fourier mass concentrated on degree 1 terms. We will then use Theorem \ref{mainchf} to bound the characteristic function of $\alpha(\k)_\Y$.
\subsection{Notation for Section and Setup}\label{setup section}
Partition the vertex set $[n]$ into $[n]=\cup_{i=0}^{k} U_i$.  Assume that for $i=1,2,\ldots,k$ all sets $U_i$ have a common size $u:=|U_i|$.  $U_0$ will contain all the other vertices, and we will always insist that $|U_0|\ge \frac{n}{k+1}$.  Once this partition has been made we can refer to a vertex in $U_i$ as having been colored with the color $i$.  Thus, tautologically, $U_i$ is the set of all vertices colored $i$.
We partition our edge variables into $k+1$ classes $B_0,B_1,\ldots, B_k$ by saying an edge $e=(u,v)$ is in $B_i$ if the largest color among the colors of its endpoints is color $i$.  Equivalently, if $e$ is an edge between a vertices in $U_i$ and $U_j$ respectively, then $e\in B_{\max(i,j)}$.

We define the $\alpha$ operator as per Section \ref{alpha def section} with respect to this partition.
For $i=1,2,\ldots,k$ let $B_i^0$ and $B_i^1$ denote two separate copies of the edges in $B_i$,
and let $Y_i^0$ and $Y_i^1$ denote independent identically drawn $p$-biased edge sets from $B_i^0$ and $B_i^1$ respectively.  Meanwhile, let $X$ denote the edge variables in $B_0={U_0\choose 2}$  Then by Lemma \ref{decoupling}
\begin{equation}\label{grestrictioneq}
|\varphi_\k(t)|^{2^k}\le \E_{\Y}\left| \E_X e^{it\sum_{\textbf{v}\in \{0,1\}^k}(-1)^{|v|} \k(X,Y^{v})}\right|=\E_\Y\left|\E_X[ e^{it\alpha( \k)_\Y(X)}]\right|
\end{equation}

%
%
We recall from Section \ref{decoupling} that $\alpha(\k)$ is a function in the variable set $X,Y_1^0,Y_1^1,\ldots,Y_k^1$.  However, the expectation we wish to bound above is only in terms of the variables in $X$.  We define our restricted functions, as per the notation in Section \ref{restrictions subsection}, to be
\begin{align}
g(X)&:=\alpha(\k)_\Y(X):=\alpha(\k)(X,\Y)=\sum_{T\subset B_0\cup \B}\hat{\k}(T)\alpha(\chi_T)(X,\Y)\nonumber\\
g_{S}(\Y)&:=g_{S}:=\widehat{\alpha(\k)_\Y}(S)=\sum_{\substack{T\subset \B\\T\rainbow}} \widehat{\k}(S\cup T) \alpha(\chi_T)(\Y)
\label{getransformeq}
\end{align}
Where in the last line $S\subset B_0$.  The rainbow condition in the subsequent sum is not technically necessary, but there to prune out the nonrainbow sets which have have a Fourier coefficient of 0.  

Thus, by equation \ref{grestrictioneq}, our goal for the rest of this section is to show that $\varphi_g(t):=\E_X[e^{itg}]$ is small with high probability over choice of restriction $\Y$.
To do this we will split the random variable $g$ into two pieces $h,d:\{0,1\}^{\B_0}\to \mathbb{R}$ where
$$h(X)=\sum_{e\in B_0}g_{e}(\Y)\chi_S(X)=g^{=1} \qquad\qquad d(X)=\sum_{|S|\ge 2} g_{S}(\Y)\chi_S(X)=g^{>1}$$

For $h$, we hope to show that its characteristic function is small, and for $d$ we will be interested in bounding $\hat{\|}d\hat{\|}_1$ and $\|d\|_2$ with an eye towards applying Theorem \ref{mainchf}.

\subsection{Main proofs of the section modulo lemmas}
The main result of this section is the following characteristic function bound:
\begin{lem}\label{puttinittogether}
Assume $\left(\frac nu \right)^{k/2}n^{k/2+\tau}\le t\le \left(\frac n u \right)^{k/2}n^{k/2+1/3-\tau}$ and $u=\Omega(n^{2\tau})$.
Then for any fixed natural number $\ell$
%
%
$$\E_{Y}|\varphi_g(t)|=O\left(n^{-\frac{\ell}{6}}\right)$$
\end{lem}

Before proving the lemma, we first state some claims, to be proven afterward, about the behavior of $g,~h,$ and $d$.

\begin{claim}\label{dsmallcor}
With probability $\ge 1-\exp(-\Omega(n^{-\tau/r^2})$ over sampling of $\Y$ we have that
\begin{align*}
\|d\|_2^2&=O\left(\left(\frac{u}{n}\right)^kn^{-k-1+2\tau}\right)\\
\hat{\|}d\hat{\|}_1&=O\left(\left(\frac{u}{n}\right)^{\frac{k}{2}}n^{-k/2+1+\tau}\right)
\end{align*}
\end{claim}

\begin{claim}\label{GeCor}
 Under the assumption that $u\ge n^{2\tau},$  $\tau<\frac12$, and $k\le r-2$.  Then there exists a constant $C$ such that for sufficiently large $n$ 
 $$\Pr\left[\left|\sum_{e\in B_0} g_e^2(\Y)\right|\le C\left(\frac{u}{n}\right)^kn^{-k}\right]\le \exp\left(-\Omega(n^{\tau/2r^2})\right)$$
\end{claim}

\begin{claim}\label{geconcentration}
For any $e\in B_0$ there exists a $C>0$ such that for all sufficiently large $n$ 
$$\Pr\left[\left|g_e(\Y)\right|\ge Cn^{-\frac k2-1+\tau}\left(\frac{u}{n}\right)^{\frac{k}{2}}\right]\le \exp\left(-\Omega\left(-n^{\tau/r^2}\right)\right)$$
\end{claim}

With these ingredients we can prove  Lemma \ref{puttinittogether}

\begin{proof}
Let $A$ be the event that $\|d\|_2^2$ and $\hat\|d\hat\|_1$ are small as promised by Claim \ref{dsmallcor}, that $g_e$ is small for all $e\in B_0$ as promised by Claim \ref{geconcentration}, and $\sum_{e\in B_0}g_e^2$ is large as promised by Claim \ref{GeCor}.
By those results we know that $\Pr(A)\ge 1-\exp(-\Omega(n^{-\tau/2r^2})$.  

We apply Theorem \ref{mainchf} to $g=h+d$.  
Conditioning on the event $A$ we can estimate the relevant parameters of that theorem to be

\begin{enumerate}
\item $T:= \sum_{e\in {U_0\choose 2}} g_e^2\ge \Omega\left(\left(\frac u n \right)^kn^{-k}\right)$ 
\item $\eta^2=\|d\|_2^2\le O\left(\left(\frac u n \right)^kn^{-k-1+2\tau}\right)$
\item $\hat{\|}d\hat{\|}_1=O\left(\left(\frac u n\right)^{\frac k 2}n^{-k/2+1+\tau}\right)$
\item $\delta^2=\max_{e}(g_e^2)=O\left(\left(\frac u n \right)^kn^{-k-2+2\tau}\right)$
\item 
\begin{align*}
\epsilon&=\exp\left(-\frac{2t^2}{\pi^2}\left[T-\delta {r\choose 2}\ell\right]\right)=\exp\left(-\Omega\left[\left(\frac u n \right)^kt^2\left(n^{-k}-\ell{r\choose 2}n^{-k-2+2\tau}\right)\right]\right)\\
&\exp\left(-\Omega\left(\left(\frac u n \right)^kt^2n^{-k}\right)\right)
\end{align*}
\end{enumerate}


So for any fixed $\ell$ such that $t\le (2e)^{\ell/2}\lambda^{-\ell/2}\eta^{-1}=O((u/n)^{-k}n^{k+1-2\tau})$ Theorem \ref{mainchf} tells us

\begin{align*}
|\varphi_g(t)-\varphi_{h}(t)|&\le \ell \epsilon\left(1+\left(t\hat{\|}d\hat{\|}_1\right)^\ell \right)+\frac{(t\eta)^{(\ell+1)}}{(\ell+1)!} \ell^{\frac{(\ell+1){r\choose 2}}{2}}\lambda^{{r\choose 2}\left(\frac{1-\ell}{2}\right)}
+\lambda^{{r\choose 2}}
\exp\left[-\frac{{r\choose 2}\lambda}{2e}|t\eta|^{-\frac{2}{{r\choose2}}}\right]\\
&\qquad+\left|t\eta\right|^{(\ell+1)/2}\lambda^{\frac{3{r\choose 2}-\ell}{4}}
\left(\ell+1\right)\exp\left[-\frac{{r\choose 2}\lambda}{4e}\left|t\eta\right|^{-\frac{2}{{r\choose 2}}}\right]\\
&=O\left(\left(\frac u n \right)^{k\ell/2}t^\ell n^{-\ell k/2+\ell+\ell \epsilon}\exp\left(-\Omega\left(\left(\frac u n \right)^kt^2n^{-k}\right)\right)+(t\eta)^{\ell+1}\right.\\
&\left.\qquad+\exp\left(-\Omega\left[(t\eta)^{-\frac{2}{{r\choose 2}}}\right]\right)+(t\eta)^{\frac{\ell+1}{2}}\exp\left(-\Omega\left[(t\eta)^{-\frac{2}{{r\choose 2}}}\right]\right)\right)
\end{align*}
Assuming that $t\ge n^{k+\tau}u^{-k/2}$
we find that the first term in the right hand side above is bounded above by $\exp\left(-\Omega(n^{2\tau})\right)$.
Additionally, assuming  $t\le u^{-k/2}n^{k+\frac13+\tau}$ we have that $t\eta=O( n^{-1/6})$, and so the subsequent terms in the above expansion
can be bounded above by $O\left(n^{-\ell/6}+\exp(-n^{-1/2r^2})\right)$.

Therefore, whenever the event $A$ occurs, we have $|\varphi_g(t)-\varphi_{g^{=1}}(t)|\le O(n^{-\ell/6})$.
Combining the fact that $\Pr(A^c)=\exp\left(\Omega(n^{-\tau/r^2})\right)$
with the observation that $|\varphi_g(t)-\varphi_{h}(t)|\le 2$, it follows that
$$\E_{\Y}|\varphi_g(t)-\varphi_{h}(t)|\le O(n^{-\ell/6})+2\Pr(A^c)=O(n^{-\ell/6})$$

To finish the proof of the lemma, we just have to bound the characteristic function of $h$.  But $h$ is a sum of independent $p$ biased Bernoulli random variables.  So we can compute (again, conditioning on $A$), that
\begin{align*}
|\varphi_{h}(t)|&=\E\left[e^{it\sum_{e\in B_0} itg_e\chi_e}\right]=\prod_{e\in B_0}\left|\E[e^{itg_e\chi_e}]\right|\le \exp\left(-\frac{4}{\pi^2}t^2\sum g_e^2\right)\le \exp(-\Omega(t^2T))\\
&\le \exp(-\Omega(n^{2\tau}))
\end{align*}
Where the first inequality uses Lemma \ref{bernoulli} combined with the assumption on the event $A$ (from Claim \ref{geconcentration}) that $|g_et|=o(1)$

\end{proof}
We now apply this Lemma for appropriate choices of $k,\ell,$ and $u$ to bound $\varphi_\k(t)$ in a form suitable for use in the proof of our main
local limit theorem in Section \ref{mainsection}
\begin{cor}\label{midchfbound}
Assume $\tau<\frac{1}{12}$.  For any $1\le k\le r-2$, and $n^{\frac k2+2\tau}\le t\le n^{\frac k 2+\frac7{12}-2\tau}$ we have $|\varphi_{\k}(t)|\le O(n^{-r^2})$.
\end{cor}
\begin{proof}
We proceed in two cases.  In the first, set $u=n/(k+1)$.  Then Lemma \ref{puttinittogether} tells us that for some constants $C_1,C_2$ we have
whenever $ C_1n^{k/2+\tau}\le t\le C_2n^{k/2+1/3-\tau}$ then $\E|\varphi_{g}(t)|=O(n^{-\ell/6})$.  Furthermore,
we know from Lemma \ref{decoupling} that $|\varphi_\k(t)|^{2^{k}}\le \E_{\Y}|\varphi_g(t)|$.  So choosing $\ell=2^{k+5}r^2$ we find that $|\varphi_\k(t)|=O(n^{-r^2})$.

In the second case set $u=n^{1-1/2k}$.  Lemma \ref{puttinittogether} along with the same choice of $\ell=2^{k+5}r^2$ will tell us that for $n^{\frac{k}{2}+\frac{1}{4}+\tau}\le  t\le n^{\frac{k}{2}+7/12-\tau}$ we have $|\varphi_\k(t)|=O(n^{-r^2})$.  So long as $\tau<\frac1{12}$ these intervals will overlap (at least in the limit). 

\end{proof}
\begin{cor}\label{midchfbound}
If $\tau<\frac{1}{12}$, then for any $t\in [n^{\frac{1}{2}+2\tau}, n^{\frac{r}{2}-\frac{5}{12}-2\tau}]$ we have $|\varphi_{\k}(t)|\le O(n^{-r^2})$.
\end{cor}
\begin{proof}
This follows by taking the union of the bounds in the above corollary for $1\le k\le r-2$.
\end{proof}

\subsection{Proof of Claim \ref{GeCor}} \label{GeCor Section}

\subsubsection{Showing that $\sum g_e^2$ is large}
In this section, we will show that, with high probability over $\Y$, $\sum_{e\in B_0} g_e^2=\|h\|_2^2$ is large.   To do this we first separate out the family
$\mathcal{F}_e$ of subsets of $\B$ which contain most of the Fourier weight of $g_e(\Y)$.
\begin{define} \label{FeDef}
$$\mathcal{F}_e=\{S\subset \B\st |\supp(S-e)|=k,~S\rainbow\}$$
\end{define}

We then carve $\sum g_e^2$ into pieces as follows.  Let
\begin{align*}
G_e(\Y)&:=\sum_{S\in \mathcal{F}_e} \widehat{g_e}(S)\alpha(\chi_S)(\Y)\\
H_e(\Y)&:=\sum_{S\notin \mathcal{F}_e} \widehat{g_e}(S)\alpha(\chi_S)(\Y)\\
Z(\Y)&=\sum_{e\in B_0} G_e^2=\sum_{e} (g_e-H_e)^2
\end{align*}
So in particular $G_e+H_e=g_e$, and $G_e$ is the main term while $H_e$ is best thought of as an error term.
We now embark on proving the following estimates for $G_e$ and $H_e$ respectively
\begin{lem}\label{GHZ}
  Assume that $1\le k\le r-2$.  Then
\begin{align*}
\|G_e\|_2^2&=\Theta\left(n^{-k-2}\left(\frac{u}{n}\right)^k\right)\\
\|H_e\|_2^2&=O\left(n^{-k-3}\left(\frac{u}{n}\right)^k\right)\\
\|g_e\|_2^2&=\Theta\left(n^{-k-2}\left(\frac{u}{n}\right)^k\right)
\end{align*}
\end{lem}
\begin{proof}
The third claim follows immediately from the previous two.  For $H_e$ we recall equation \ref{getransformeq} and compute.
\begin{align*}
2^{-k}\|H_e\|_2^2=2^{-k}\sum_{S\in \mathcal{F}_e^c} \widehat{g_e}(S)^2&=\sum_{t=k+3}^{r} \sum_{\substack{|V\cup e|=t}}\sum_{\substack{S\subset{V+e\choose 2}\\S\rainbow}} \hat{\k}(S\cup e)^2\le\sum_{t=k+3}^{r} \sum_{|V|=t}\sum_{\substack{S\subset{V+e\choose 2}\\S\rainbow}} 
O\left(n^{-2t+2}\right)\\
&\le\sum_{t=k+3}^r u^kn^{t-k-2}2^{t\choose 2} O(n^{-2t+2})\le O\left(n^{-k-3}\left(\frac{u}{n}\right)^k\right)
\end{align*}
Where the second inequality follows from noting that there are at most $u^{k}n^{|\supp(S)|-k-2}$ ways to choose the support of a rainbow set of edges,
and then at most $2^{|\supp(S)|\choose 2}$ ways to pick edges with that support.

Meanwhile for $G_e$ 
\begin{align*}
2^{-k}\E[G_e(\Y)^2]&=\sum_{S\in \mathcal{F}_e} \hat{\k}(S\cup e)^2= \sum_{|V|=k+2}\sum_{\substack{\supp(S\cup e)=V\\S\rainbow}} \hat{\k}(S\cup e)^2= \sum_{|V|=k+2}\sum_{\substack{\supp(S\cup e)=V\\S\rainbow}} C_Sn^{-2k-2}\\
&\ge \left(\prod_{i=1}^k |U_i|\right)C_Sn^{-2k-2}\\
&=\Theta\left(n^{-k-2}\left(\frac{u}{n}\right)^k\right)
\end{align*}
Where $C_S$ is some constant depending on $|S|$, which always lies in $[\lambda^{r^2},1]$ and can be read off of the Fourier expansion of $\k$ in Section \ref{KrProp} .  Note that we used the assumption that $k+2\le r$ to ensure that the sums above were nonempty.
\end{proof}

Meanwhile, for each $e\in U$ we know that $|G_e-g_e|=|H_e|$, and so
 $g_e^2\ge G_e^2-2|G_eH_e|$.  But $H_e$ is relatively small, so by Cauchy Schwarz
$$\|G_eH_e\|_2^2=\E[G_e^2H_e^2]\le \sqrt{\E[G_e^4]\E[H_e^4]}=\|G_e\|_4^2\|H_e\|_4^2$$
Since $\deg(G_e),\deg(H_e)\le {r\choose 2}$  Theorem \ref{moment hypercontractivity} tells us that
$$\|G_e\|_4^4\|H_e\|_4^4\le (3)^{2r^2}\lambda^{2r^2-4}\|H_e\|_2^4\|G_e\|_2^4=\Theta\left((u/n)^{4k}n^{-4k-10}\right)$$
This in turn implies that $\|G_eH_e\|_2=\Theta(u^{k}n^{-2k-2.5})$.  So it follows from Theorem \ref{concentration hypercontractivity} that $|G_eH_e|\le u^{k}n^{-2.5+\tau}$ with probability $1-\exp(\Omega(-n^{\tau/2r^2})$ .  Therefore with high probability
$$\sum_e g_e^2\ge \sum_{e\in B_0} \left[G_e^2-O(u^{k}n^{-2k-2.5+\tau})\right]=Z-O(u^{k}n^{-2k-1/2+\tau})$$
 We restate this as a lemma.
 \begin{lem}\label{ZtoGBound}
 With probability at least $1-\exp(\Omega(n^{-\tau/2r^2})$ over choice of $\Y$, we have that
 $$\sum_e g_e^2\ge \sum_{e\in B_0} \left[G_e^2-O(u^{k}n^{-2k-2.5+\tau})\right]=Z-O(u^{k}n^{-2k-\tau+1/2})$$
 \end{lem}
 
 To finish our argument, we require the fact that $Z$ is large with high probability.  This will follow immediately from observing that
 $Z$ is a fixed degree polynomial and computing the variance of $Z$.  Unfortunately, computing this variance is cumbersome, and so
 the proof of the following lemma is in Appendix \ref{appendix z}
 \begin{lem}\label{GeBound}
 Let $Z=\sum_{e\in B_0} G_e^2$.  Assume that $u\ge n^{2\tau}$  and $1\le k\le r-2$.  Then there exists a constant $C$ such that for sufficiently large $n$, $\Pr\left[|Z|\le C\left(\frac{u}{n}\right)^kn^{-k}\right]\le \exp\left(\Omega(n^{-\tau/r^2})\right)$
 \end{lem}
 
 We are now in a position to prove Claim \ref{GeCor}

\begin{repclaim}{GeCor}
 Under the assumption that $u\ge n^{2\tau},$  $\tau<\frac12$, and $k\le r-2$.  Then there exists a constant $C$ such that for sufficiently large $n$ 
 $$\Pr\left[\left|\sum_{e\in B_0} g_e^2(\Y)\right|\le C\left(\frac{u}{n}\right)^kn^{-k}\right]\le \exp\left(\Omega(n^{-\tau/2r^2})\right)$$
\end{repclaim}
\begin{proof}
Lemma \ref{GeBound} above implies that  for some $C_1>0$ we have $Z\ge C_1u^kn^{-2k}$ with probability  $1-\exp\left(\Omega(n^{-\tau/r^2})\right)$.
Meanwhile Lemma \ref{ZtoGBound} implies that $\sum_{e\in B_0} g_e^2\ge Z-O(u^kn^{-2k-\tau+\frac{1}{2}})$ with probability  $1-\exp\left(\Omega(n^{-\tau/2r^2})\right)$.  Combining these inequalities yields the corollary.
\end{proof}

\subsection{Proof of Claim \ref{geconcentration}}

\begin{repclaim}{geconcentration}
For any $e\in B_0$ there exists a $C>0$ such that for all sufficiently large $n$ 
$$\Pr\left[\left|g_e(\Y)\right|\ge Cn^{-\frac k2-1+\tau}\left(\frac{u}{n}\right)^{\frac{k}{2}}\right]\le \exp\left(-\Theta\left(n^{-\frac{\tau}{r^2}}\right)\right)$$
\end{repclaim}
\begin{proof}
We computed in Lemma \ref{GHZ} that $\|g_e\|_2=O(n^{-k-2}(u/n)^k)$.  We also know that $\E[g_e(\Y)]=\hat{g_e}(\varnothing)=0$.  $g_e$ is a polynomial in $\Y$ of degree at most $r\choose 2$ so the 
result then follows from Lemma \ref{concentration hypercontractivity}.
\end{proof}

\subsection{Proof of Claim \ref{dsmallcor}}

First, we compute a bound on the Fourier coefficients $g_S(\Y)=\widehat{\alpha(\k)_\Y}(S)$.
\begin{lem}\label{d2union}
For some $C>0$ and for all sufficienetly large $n$ 
$$g_{S}(\Y)^2\le \left(\frac{u}{n}\right)^kn^{-k-2|\supp(S)|+2+2\tau}$$
 holds for \textit{all} $S\subset B_0$ with probability $1-\exp(-\Theta(n^{\tau/r^2}))$.
\end{lem}
\begin{proof}
For any $S$, we note that $\E_{\Y}[g_S(\Y)]=0$.  If $|S|>r-k$ then $g_S(\Y)$ is identically 0.  If $|S|\le r-k$ then we compute this quantity to have variance
\begin{align*}
2^{-k}\E[g_{S}(\Y)^2]&=\sum_{\substack{T\subset \B\\T\rainbow}}  \hat{\k}(S\cup T)^2=\sum_{t=k+|\supp(S)|}^{r} \sum_{|V|=t}\sum_{\substack{\supp(S\cup T)=V\\T\rainbow}} \hat{\k}(S\cup T)^2\\
&\le\sum_{t=k+|\supp(S)|}^{r} \sum_{|V|=t}\sum_{\substack{\supp(S\cup T)=V\\T\rainbow}} O\left(n^{-2t+2}\right)\\
&\le\sum_{t=k+|\supp(S)|}^{r}  u^kn^{t-k-|\supp(S)|}2^{t\choose 2}O\left(n^{-2t+2}\right)=O\left(\left(\frac{u}{n}\right)^kn^{-k-2|\supp(S)|+2}\right)
\end{align*}
We also know that $g_{S}(\Y)$ is a polyomial of degree at most ${r\choose 2}$ and $\E[g_{\B|S}(\Y)]=0$.  By Theorem \ref{concentration hypercontractivity},  for $n$ sufficiently large we have that $\Pr[|g_{S}|\ge \left(\frac{u}{n}\right)^{k/2}n^{-k/2-|\supp(S)|+1+\tau}]\le \exp(-\Omega(n^{\tau/r^2}))$.  Since there are at most $O(n^{r^2})$ possible monomials $S$
for which $g_S(\Y)\not \equiv 0 $ a union bound finishes the proof of the lemma.
\end{proof}
\begin{repclaim}{dsmallcor}
With probability $1-\exp(-\Omega(n^{-\tau/r^2})$ we have that  both
\begin{align*}
\|d\|_2^2&=O\left(\left(\frac{u}{n}\right)^kn^{-k-1+2\tau}\right)\\
\hat{\|}d\hat{\|}_1&=O\left(\left(\frac{u}{n}\right)^{\frac{k}{2}}n^{-k/2+1+\tau}\right)
\end{align*}
\end{repclaim}
\begin{proof}
Both of these statements follow from a computation using the bound given in Lemma \ref{d2union}.  Throughout we condition on the assumption that all of the Fourier coeficients of $d$ are as small as promised by Lemma \ref{d2union}.  First we bound  the 2-norm of $d$ by
\begin{align*}
\|d\|_2^2&=\sum_{\substack{S\subset B_0\\|S|\ge 2}} g_S(\Y)^2=\sum_{t=3}^{r-k}\sum_{|\supp(S)|=t} g_S(\Y)^2\le\left(\frac{u}{n}\right)^k\sum_{t=3}^{r-k}\sum_{|\supp(S)|=t} n^{-k-2t+2+2\tau}\\
&\le\left(\frac{u}{n}\right)^k\sum_{t=3}^{r-k}{n\choose t}2^tO\left(n^{-k-2t+2+2\tau}\right)=O\left(\left(\frac{u}{n}\right)^kn^{-k-1+2\tau}\right)
\end{align*}
Then we bound the spectral 1 norm by
\begin{align*}
\hat{\|}d\hat{\|}_1&=\sum_{S\subset B_0}| \hat{d}(S)|=\sum_{t=3}^{r-k} \sum_{|\supp(S)|=t} O\left(\left(\frac{u}{n}\right)^{k/2}n^{-k/2-t+1+\tau}\right)=O\left(\left(\frac{u}{n}\right)^{k/2}n^{-k/2+1+\tau}\right)
\end{align*}
\end{proof}

\section{Bound for $|t|\in[ n^{\frac{r}{2}-\frac16-2\tau},~n^{r-1-r\tau}]$}
Here we repeat the same setup and notation from Section \ref{setup section}, but now we focus exclusively on the special case when $k=r-2$.
The function $g=\alpha(\k)_\Y$ exhibits some different behavior in this case. 

First, let's look at what happens when $k=r-1$.  Then any rainbow set $T\subset \B$ contains vertices of from $U_1,U_2,\ldots,U_k$, and so in particular has at least $r-1$ vertices not in $U_0$.  Therefore for any nonempty $S\subset {U_0\choose 2}$ we have $|\supp(S\cup T)|\ge r+1$.  But recall that $\k$ is supported on sets of edges spanning at most $r$ vertices.  Therefore $\alpha(\k)$ does not depend on the edges in $B_0$ at all, and so $g(X)$ is a constant.

But if $k=r-2$, then for any rainbow $T\subset \B$ and $S\subset{U_0\choose 2}$ if $|S|>1$ we have
$|\supp(S\cup T)|\ge k+3=r+1$.  Therefore $\hat{g}(S)=0$ when $S$ is not just the set of a singleton edge.  In particular for any edge $e\in {U_0\choose 2}$ if $T$ is rainbow, and $|\supp(e\cup T)|\le r$ then it follows that $|\supp(T)-e|=k=r-2$.  Recallling definition \ref{FeDef} that
$\mathcal{F}_e=\{S\subset \B\st |\supp(S)-e|=k,~S\mbox{ rainbow}\}$, we can restate our observation as the following lemma.
\begin{lem}
Assume $k=r-2$.  Then we have $d\equiv 0$.  Further, for $e\in {U_0\choose 2}$ and $T\subset \B,~T\notin \mathcal{F}_e$, it follows that $\widehat{g_{e}}(T)\equiv 0$.  That is $g_e=\sum_{T\in \mathcal{F}_e} \widehat{g_{e}}(T)\chi_T(\Y)$.
\end{lem}

This implies that for any choice of $\Y\in \{0,1\}^{\B}$ we have $\alpha(\k)_\Y=g$ is a a degree 1 polynomial in independent Bernoulli random variables.
Because of this, we can bound the characteristic function of $g$ more directly.  Additionally, our analysis of $\sum g_e^2=\|g\|_2^2$ also becomes easier.

\begin{lem}\label{lowhighchfbound}
For $t\in [ (r-1)^{r/2-1}n^{r/2-1+\tau/2},~n^{r-2-r\tau}]$, we have $|\E_{\Y}[\varphi_g(t)]|=\exp(-\Omega(n^{\tau/2r^2}))$.
\end{lem}
\begin{proof}
Set $u=n^{2+\tau/(r-2)}t^{-2/(r-2)}$, and therefore $(u/n)^{r-2}=n^{r-2+\tau}/t^2$.  First, we check that this is a feasible choice of $u$ for Claims \ref{GeCor} and \ref{geconcentration}, that is $n^{2\tau}\le u\le n/(r-1)$.  
For the lower bound, we find the requirement 
$$n^{2+\tau/(r-2)}t^{-2/(r-2)}\ge n^{2\tau}\implies t \le n^{r-2-r\tau}$$
For the upper bound we need
$$n^{2+\tau/(r-2)}t^{-2/(r-2)}\le \frac{n}{r-1}\implies t \ge (r-1)^{r/2-1}n^{r/2-1+\tau/2}$$
And these are exactly the hypotheses on $t$.
Let $A$ be the event that $\sum_{e\in B_0} g_e^2\ge (u/n)^{r-2}n^{-r+2}=n^{\tau}/t^2$, and $g_e^2\le (u/n)^{r-2}n^{-r+\tau}=n^{-2+2\tau}/t^2$ for all $e\in B_0$.  By Claims \ref{GeCor} and \ref{geconcentration} we have $\Pr(A^c)\le \exp(-\Omega(n^{\tau/2r^2}))$.

Given that event $A$ occurs we have that $|g_e t|\le n^{-1+\tau}<\sqrt{p(1-p)}\pi$.  Therefore we can use Lemma \ref{bernoulli} to bound 

$$|\varphi_g(t)|=|\E[e^{it\sum g_e\chi_e}]|=\prod_{e\in B_0} \E[e^{itg_e\chi_e}]\le e^{-\frac{2}{\pi^2}t^2\sum g_e^2}\le \exp(-\Omega(n^{\tau/r^2}))$$.

To complete the proof of the lemma, we combine this inequality with the bounds that $\Pr(A^c)\le \exp(-\Omega(n^{-\tau/2r^2}))$ and $|e^{ix}|\le 1$.
\end{proof}

Setting $u$ slightly smaller yields a result more suited to slightly larger values of $t$
\begin{lem} \label{highmidchfbound}
For $t\in [ (r-1)^{r/2-1}n^{r/2-3\tau/2},~n^{r-1-2(r-2)\tau}]$, we have $|\E_{\Y}[\varphi_g(t)]|=\exp(-\Omega(n^{\tau/2r^2}))$.
\end{lem}
The proof is more or less the same as the above, but we include it here as subtle errors would be easy to make.
\begin{proof}
Set $u=n^{2+\frac{2-3\tau}{r-2}}t^{-2/(r-2)}$, and therefore $(u/n)^{r-2}=n^{r-3\tau}/t^2$.  First, we check that this is a feasible choice of $u$ for Claims \ref{GeCor} and \ref{geconcentration}, that is $n^{2\tau}\le u\le n/(r-1)$.  
For the lower bound, we find the requirement 
$$n^{2+\frac{2-3\tau}{r-2}}t^{-2/(r-2)}\ge n^{2\tau}\implies t \le n^{r-1-2(r-2)\tau}$$
For the upper bound we need
$$n^{2+\frac{2-3\tau}{r-2}}t^{-2/(r-2)}\le \frac{n}{r-1}\implies t \ge (r-1)^{r/2-1}n^{r/2-3\tau/2}$$
And these are exactly the hypotheses on $t$.
Let $A$ be the event that $\sum_{e\in B_0} g_e^2\ge (u/n)^{r-2}n^{-r+2}=n^{2-3\tau}/t^2$, and $g_e^2\le (u/n)^{r-2}n^{-r+2\tau}=n^{-\tau}/t^2$ for all $e\in B_0$.  By Claims \ref{GeCor} and \ref{geconcentration} respectively we have $\Pr(A^c)\le \exp(-n^{\tau/2r^2})$.

Given that event $A$ occurs we have that, $|g_e t|\le n^{-\tau/2}<\sqrt{p(1-p)}\pi$.  Therefore we can use Lemma \ref{bernoulli} to bound (again
conditional on the event $A$ occuring)

$$|\varphi_g(t)|=|\E[e^{it\sum g_e\chi_e}]|=\prod_{e\in B_0} \E[e^{itg_e\chi_e}]\le e^{-\frac{4}{\pi^2}t^2\sum g_e^2}\le \exp(-\Omega(n^{2-3\tau}))$$.

To complete the proof of the lemma, just use the fact the bound $\Pr(A^c)\le \exp(-\Omega(n^{-\tau/2r^2}))$ and the fact that $|e^{ix}|\le 1$.
\end{proof}

Combining Lemmas \ref{lowhighchfbound} and \ref{highmidchfbound} with Lemma \ref{decoupling} yields the following corollary
\begin{cor}\label{midhighchfbound}
Assume $0<\tau<\frac1{2r}$.  For $t\in [n^{r/2-1+\tau},~n^{r-1-2(r-2)\tau}]$ we have $|\varphi_\k(t)|\le \exp(-\Omega(n^{\tau/2r^2}))$.
\end{cor}

\section{Bound for large $t$} \label{larget section}
For large $t$, an even more extreme application of Lemma \ref{decoupling} is needed.  To do this, we take the following partition of the edge random variables.
Partition the vertex set $[n]$ into $\lfloor \frac{n}{r}\rfloor$ $r$-cliques.  Let $\mathcal{F}$ be the family of cliques in this partition.   Now let $\tilde{B_0},B_1,\ldots, B_{{r\choose 2}-1}$ be any partition of the edges 
of these cliques such that each $B_i$ contains \textit{exactly one} edge from each clique in $\mathcal{F}$.  Now set $B_0$ to be the union of $\tilde{B_0}$ along with all
edges of $K_n$ not already partitioned into a $B_i$ (i.e., edges connecting the different cliques in $\mathcal{F}$ as well as the leftover edges from vertices not put into cliques).  See Figure \ref{larget figure} for an example of this partition.
In this section, rather than using the orthogonal character functions, it will be more convenient to use indicator vectors $x_e\in \{0,1\}$
instead.  Additionally for a set of edges $S$, we will use $x^S$ to denote the monomial $\prod_{e\in S} x_e$.

Let $X\in \{0,1\}^{B_0}$ and $Y_i^0,~Y_i^1\in \{0,1\}^{B_i}$ independent as in section \ref{decoupling section}.
As before, for a given setting of $\Y\in \{0,1\}^\B$  we define $g(X)$ by setting
\begin{align*}
g(X):=\alpha(\k)_\Y(X)=\alpha(\k)(X,\Y)
\end{align*}
Recall that $\alpha$ is a linear operator, and that furthermore we have $\alpha(x^S)=0$ unless $S$ is a rainbow set of edges.  However, by construction we know that the only rainbow sets $S$ are exactly the cliques $S\in \mathcal{F}$.  Therefore we have
$$\alpha(\k)=\sum_{S\equiv K_r} \alpha(x^S)=\sum_{S\in \mathcal{F}} \alpha(x^S)$$

For each $S\in \mathcal{F}$, we have $S=e\cup S'$ where $e\in B_0$ and $S'\subset \cup_{i\ge 1} B_i$.  So, for any fixed $S\in \mathcal{F}$, if we sample $\Y$ at random, then we have that $Y_{e'}^0=1$ and $Y_{e'}^1=0$ for all edges $e'\in S'$ with probability at least $\lambda^{2{r\choose 2}^2}$.\footnote{recall $\lambda=\min(p,1-p)$}  Label this event $A_S$.  If $A_S$ occurs then it follows that
$$\alpha(x^S)(X,Y)=x_e\sum_{v\in \{0,1\}^k} (-1)^{|v|} x^{S'}=x_e$$
as the only nonzero term in the above sum is when $v=(0,0,\ldots,0)$.
Given that $A_S$ occurs and $t\le \pi \sigma$ by Lemma \ref{bernoulli} we have
$$\E_{x_e} e^{itx^S/\sigma}=\E_{x_e} e^{itx_e/\sigma}\le 1-\frac{8p(1-p)t^2}{\pi^2\sigma^2}$$

Let $z(\Y)$ denote the number of edges $e\in \tilde B_0$ such that $A_e$ occurs.  Using the fact that $g=\sum_{\mathcal{F}} \alpha(x^S)$ and that 
the random variables $\alpha(x^S)$ are independent, we may compute
$$|\E[e^{itg}]|=\prod_{S\in A}\left|\E[e^{it\alpha(x^S)}]\right|\le \prod_{\substack{S\in A\\A_S\mbox{\scriptsize{ occurs}}}} \left(1-\frac{8p(1-p)t^2}{\pi^2\sigma^2}
\right)=\exp\left(-\frac{8p(1-p)t^2z(Y)}{\pi^2\sigma^2}\right)$$
Since each of the events $A_S$ are independent and occur with probability $\ge\lambda^{2{r\choose 2}^2}$ it follows from Chernoff bounds that $z(Y)\ge \lambda^{2{r\choose 2}^2}n/2r$ with probability  $\ge 1-\exp(-\Omega(n))$. 

So we find that
$$\E_\Y \left|\E_X[e^{it\alpha(\k)}]\right|\le \Pr(A) \exp\left(-\frac{8p(1-p)t^2}{\pi^2\sigma^2}\cdot\frac{\lambda^{2{r\choose 2}^2}n}{2r}\right)+\Pr(A^c)=\exp\left(-\Omega\left(\frac{t^2n}{\sigma^2}\right)\right)$$

Combining this with Lemma \ref{decoupling} we have proved the following:

\begin{lem}\label{highchfbound}
For $|t|\le \pi \sigma$ we have that $|\varphi_{\k}(t)|\le \exp(-\Omega(t^2n/\sigma^2))$.
\end{lem}

\begin{figure}[H]\label{larget figure}
\begin{centering}
\begin{tikzpicture}
\foreach \i in {1,2}{
	\foreach \j in {1,2,3}{
		\node[draw, circle, minimum size = .75cm, inner sep = 0cm, xshift = 4*\i cm]  (v\i\j) at ({(\j-1)*120-30}:1cm){$v_{\i\j}$};
	}
}	
\foreach \i in {1,2}{
	\draw[loosely dashed, ultra thick] (v\i1)--(v\i2);		
	\draw[dotted, ultra thick] (v\i3)--(v\i2);	
	\draw[ thin] (v\i1)--(v\i3);									
}
%
%
%
%
%
%
%
%
%

\foreach \j in {1,2}{
	\foreach \k in {2,3}{
		\draw[ thin] (v1\j) to (v2\k);
	}
	
	\draw[ thin](v13) to (v22);
	\draw[ thin, bend right = 30](v13) to (v23);
	\draw[ thin, bend right = 40](v13) to (v21);
	\draw[ thin, bend right = 20](v11) to (v21);
}

\end{tikzpicture}
\caption{Example illustrating the partition $B_0,B_1,\ldots, B_{r\choose 2-1}$ from section \ref{larget section}.  In this case (where $r=3$) we have
3 line styles (loosely dahsed, dotted, and thin) representing the 3 different edge classes.  Note that the only rainbow triangles are $(v_{11},v_{12},v_{13})$ and $(v_{21},v_{22},v_{23})$.
$B_0$, here represented by the thin edges is quite large, but most of the edges are not on even a single rainbow triangle and so $g(X)$ does
not depend on them at all}
\end{centering}
\end{figure}

\bibliography{TryNBib}
\bibliographystyle{alpha}

\appendix
\section{Proof of Lemma \ref{alpha product}} \label{alpha product appendix}
\begin{replem}{alpha product}
 Let $S,T\subset {[n]\choose 2}$.  For $0\le i\le k$ let $S_i=S\cap B_i$ and $T_i=T\cap B_i$.
$$\alpha(\chi_S)\alpha(\chi_T)=\left(\prod_{i=1}^k\left[ \sum_{\iota =0}^1-\chi_{S_i^\iota}\chi_{T_i^{\iota\oplus1}}+\chi_{S_i^\iota\Delta T_i^\iota}\left(\sum_{U_i\subset S_i^\iota\cap T_i^\iota} \gamma^{|U_i|}\chi_{U_i}\right) \right]\right)$$
In particular, for $U\subset B_0\cup \B$ we have
\begin{align*}
\left|\widehat{\alpha(\chi_{S})\alpha(\chi_T)}(U)\right|\le \begin{cases}
0&\mbox{if }S\Delta T\not \subset \fl(U)\mbox{ or } \fl(U)\not \subset S\cup T\\
\max\left(\gamma^{|U|}, 1\right)&\mbox{if } S\Delta T\subset \fl(U)\subset S\cup T
\end{cases}
\end{align*}
\end{replem}

\begin{proof}

For sets $S,T\subset \B$ we compute
\begin{align*}
\alpha(\chi_{S})\alpha(\chi_T)&=\left(\prod_{i=1}^k\left[\chi_{S_i}(Y_i^{0})-\chi_{S_i}(Y_i^{1})\right]\right)\left(\prod_{i=1}^k\left[\chi_{T_i}(Y_i^{0})-\chi_{T_i}(Y_i^{1})\right]\right)\\
&=\left(\prod_{i=1}^k\left[- \chi_{S_i}(Y_i^0)\chi_{T_i}(Y_i^1)-\chi_{S_i}(Y_i^1)\chi_{T_i}(Y_i^0)+\chi_{S_i}(Y_i^0)\chi_{T_i}(Y_i^0)+\chi_{S_i}(Y_i^1)\chi_{T_i}(Y_i^1) \right]\right)\\
&=\left(\prod_{i=1}^k\left[ \sum_{\iota =0}^1-\chi_{S_i}(Y_i^\iota)\chi_{T_i}(Y_i^{\iota\oplus1})+\chi_{S_i}(Y_i^\iota)\chi_{T_i}(Y_i^\iota) \right]\right)\\
&=\left(\prod_{i=1}^k\left[ \sum_{\iota =0}^1-\chi_{S_i}(Y_i^\iota)\chi_{T_i}(Y_i^{\iota\oplus1})+\chi_{S_i\Delta T_i}(Y_i^\iota)\prod_{e\in {S\cap T}}\left(1+\frac{1-2p}{\sqrt{p(1-p)}}\chi_e(Y_i^\iota)\right) \right]\right)\\
&=\left(\prod_{i=1}^k\left[ \sum_{\iota =0}^1-\chi_{S_i}(Y_i^\iota)\chi_{T_i}(Y_i^{\iota\oplus1})+\chi_{S_i\Delta T_i}(Y_i^\iota)\left(\sum_{U_i\subset S_i\cap T_i} \gamma^{|U_i|}\chi_{U_i}(Y_i^\iota)\right) \right]\right)\\
&=\left(\prod_{i=1}^k\left[ \sum_{\iota =0}^1-\chi_{S_i^\iota}\chi_{T_i^{\iota\oplus1}}+\chi_{S_i^\iota\Delta T_i^\iota}\left(\sum_{U_i\subset S_i^\iota\cap T_i^\iota} \gamma^{|U_i|}\chi_{U_i}\right) \right]\right)
\end{align*}

So we see that this is supported only on sets $U$ such that for each $i\in [k]$ we have that $U_i=S_i^\iota\cup T_i^{\iota\oplus 1}$ or $S_i\Delta T_i\subset U_i$.

In both of these cases $S_i\Delta T_i\subset\fl(U_i)\subset S_i\cup T_i$.  
Furthermore each of the terms appearing in the expansion of the product in the RHS of the above are unique, and each has coefficients of bounded size.  In particular for any set $U$ we find that

\begin{align*}
\left|\widehat{\alpha(\chi_{S})\alpha(\chi_T)}(V)\right|\le \begin{cases}
0&\mbox{if }S\Delta T\not \subset \fl(V)\\
\max\left(\gamma^{|V|}, 1\right)&\mbox{if } S\Delta T\subset \fl(V)
\end{cases}
\end{align*}
\end{proof}

\section{Proof of Lemma \ref{GeBound}}\label{appendix z}
In this section we prove Lemma \ref{GeBound} from Section \ref{GeCor Section}.\footnote{All terminology and parameters should be set as in that section, and are not repeated here.}

\begin{replem}{GeBound}
Let $Z=\sum_{e\in B_0} G_e^2$.  Assume that $u\ge n^{2\tau}$  and $1\le k\le r-2$.  Then there exists a constant $C_0$ such that for sufficiently large $n$, $\Pr\left[|Z|\le C_0\left(\frac{u}{n}\right)^kn^{-k}\right]\le \exp\left(\Omega(n^{-\tau/r^2})\right)$
\end{replem}

To build to this lemma we first analyze the transform of each summand $G_e^2$ individually.  $\hat{G_e}$ is supported on sets $S\in \mathcal{F}_e$, that is, sets $S$ such that $\supp(S)-\supp(e)$ consists of 1 vertex from each color class $U_1,U_2,\ldots, U_k$. 
For sets $S,T\in \mathcal{F}_e$  Lemma \ref{alpha product} tells us that
\begin{align*}
\left|\widehat{\alpha(\chi_{S})\alpha(\chi_T)}(V)\right|\le \begin{cases}
0&\mbox{if }S\Delta T\not \subset \fl(V)\mbox{ or } \fl(V)\not \subset S\cup T\\
C&\mbox{if } S\Delta T\subset \fl(V)\subset S\cup T
\end{cases}
\end{align*}
where $C$ is some constant depending only on $r$ and $p$.

Therefore, for $V\subset \B$ we can bound the Fourier coefficients of $G_e^2$ by
\begin{align}\label{GeCoeffBound}
\widehat{G_e^2}(V)=\sum_{S,T\in \mathcal{F}_e}\hat{g_e}(S)\hat{g_e}(T)\widehat{\alpha(\chi_S)\alpha(\chi_T)}(V)\le \sum_{\substack{S,T\in \mathcal{F}_e\\S\Delta T\subset\fl(V)\subset S\cup T}}C^2 \hat{\k}(S\cup e)\hat{\k}(T\cup e)
\end{align}
For $S\in \mathcal{F}_e$, we know that  $\hat{K}(S\cup e)=\Theta(n^{-|\supp(S\cup e)|+1})=\Theta(n^{-k-1})$.  So the above sum can be reduced to a counting problem.  For any given
set $V\subset \B$ we need to count the number of pairs $S.T\in \mathcal{F}_e$ such that $S\Delta T\subset \fl(V)\subset S\cup T$.
First, to help with this counting problem we define the auxiliary color function to be
$c(S)=\{i \st \supp(S)\cap U_i\neq \varnothing~\mbox{and }1\le i\le k$\}.  That is $c(S)$ is the number of the vertex partitions $U_1,U_2,\ldots, U_k$
that $S$ sees.  A few helpful observations:  
\begin{itemize}
\item Touching a vertex in $U_0$ is not counted in $c(V)$
\item For any $S\in \mathcal{F}_e$ we have $c(S)=k$
\item The above bullet is not true for all sets in the spectrum of $G_e^2$, as for example the empty set has a nontrivial coefficient of $\widehat{G_e^2}(\varnothing)=
\E[G_e^2]=\Theta((u/n)^kn^{-k-2})$
\end{itemize}

We solve this counting problem with the following Lemma
\begin{lem}\label{CountingLemma}
Fix some $V\subset \B$ and let $\A=\{(S,T)\subset \mathcal{F}_e^2\st S\Delta T\subset\fl(V)\subset S\cup T\}$.  If $\supp(V)\cap U_0\not \subset \supp(e)$, then $|\A|=0$.  Otherwise, $|\A|\le O(u^{k-c(V)})$.
\end{lem}
\begin{proof}
The first claim follows from just noting that $S,T\in \mathcal{F}_e$ implies that $\supp(S)\cap U_0\subset \supp(e)$ and $\supp(T)\cap U_0\subset \supp(e)$.

For the second inequality, we first count the number of ways to choose the support of $S\cup T$.  Since $S\Delta T\subset \fl(V)$, it follows that $\supp(S\Delta T)\subset \supp(\fl(V))$.
However for any set of edges $S,T$ it must be that $\supp(S)\Delta\supp(T)\subset \supp(S\Delta T)$.  Therefore $\supp(S)\Delta \supp(T)\subset \supp(V)$.

Now we establish what $\supp(S)\cap U_i$ and $\supp(T)\cap U_i$ can look like, depending on the properties of $V$.  We do this in three cases:
\begin{enumerate}
\item $|\supp(V)\cap U_i|=1$.  Let $v_i$ be the vertex in the intersection.  Because $\supp(S)\Delta\supp(T)\subset \supp(V)$ and $|\supp(S)\cap U_i|=|\supp(T)\cap U_i|=1$ it follows that $\supp(S)\cap U_i=\supp(T)\cap U_i=v_i$.
\item $|\supp(V)\cap U_i|\ge 2$.  Since $S,T$ are both rainbow, we have that $S\cup T$ is supported on at most 2 vertices of color $i$.  Combining this with the hypothesis that $\supp(V)\subset \supp(S)\cup \supp(T)$ yields that $\supp(V)\cap U_i=\supp(S\cup T)\cap U_i$.
\item  $|\supp(V)\cap U_i|=0$.  The above observation that $\supp(S)\Delta \supp(T)\subset \supp(V)$ tells us that $(\supp(S)\Delta\supp(T))\cap U_i=\varnothing$ and so $\supp(S)\cap U_i=\supp(T)\cap U_i$.  This single point of intersection could be any arbitary point from $U_i$ leaving at most $u$ possible choices.
\end{enumerate}
%

So combining all these cases we find that 
\begin{itemize}
\item If $|\supp(V)\cap U_i|\ge 1$, then $\supp(S\cup T)\cap U_i$ is determined uniquely by $V$
\item If $|\supp(V)\cap U_i|=0$ then we permit that $\supp(S\cup T)\cap U_i$ could be any one vertex from $U_i$
\end{itemize}
Meanwhile $\supp(S\cup T)\cap U_0\subset e$, and so there are at most 4 possible choices for this set.
Combining this over all possible choices, we see that $\supp(S\cup T)$  can take at most $4u^{k-c(v)}$ distinct possible values.
%
Lastly we note that there are at most $\Theta(1)$ possible ways to decide how to form rainbow sets of edges $S,T$ supported on a fixed set of
at most $2k+2$ vertices, so this finishes the proof.
\end{proof}

\begin{lem}\label{CD}
For $j\in \{0,1,2\}$ and $0\le c\le k$ define
$$\C_j:=\{ V\subset \B| \widehat{Z}(V)\neq 0,~c(V)=c,~|\supp(V)\cap U_0|=j\}$$
Then:
\begin{align*}
|\C_0|\le u^{2c-1}\qquad|\C_1|\le nu^{2c}\qquad
|\C_2|\le n^2u^{2c}
\end{align*}

\end{lem}

\begin{proof}
We prove the claims on the size of $\C_1, \C_2$ first.  Any set $V$ such that $\hat{Z}(V)\neq 0$ has the property that for some $S,T\in \mathcal{F}_e$ we have $\fl(V)\subset S\cup T$.
Since any $S,T$ each contain exactly one vertex in each of the $U_i$, it follows that $S\cup T$ is supported on at most 2 vertices in each $U_i$.  Furthermore, for $V\in \C_j$ there are at most $n^j$ possible choices for $\supp(V)\cap U_0$.  Therefore
 there are at most $u^{2c(V)}n^j$ ways to choose the vertex support of $V$.  Once the vertices are chosen, there are only $O(1)$ subsets of edges from
$\B$ supported on any vertex set of size at most $2k$.  This shows our bound on $|\C_1|$ and $|\C_2|$.

For the claim about $|\C_0|$, note that if $V\in \C_0$, then as above there are some $S,T\in \mathcal{F}_e$ such that $S\Delta T\subset \fl(V)$.  Now let $i$ be the smallest index such that $V$ is supported on a vertex in $U_i$. 
$S$ has an edge of the form $(a,b)$ where $a\in U_i$ and $b\in U_j$ for some $j<i$ by the rainbow condition for membership in $\mathcal{F}_e$.
It follows that $b\notin \supp(V)$, and hence $(a,b)\notin \fl(V)$.
Hence it must be the case that $(a,b)\in T$ as well.  So $S\cap U_i=T\cap U_i=a$.  Therefore $\supp(V)$ contains at most 1 vertex from $U_i$.  Continuing
to count the number of  possible choices of the support of $V$ as we did above we find that
$|\C_0|=O(u^{2c-1})$.
\end{proof}

From these two lemmas we can obtain our desired concentration bound for $\sum_{e} G_e^2$.
\begin{replem}{GeBound}
Let $Z=\sum_{e\in B_0} G_e^2$.  Assume that $u\ge n^{2\tau}$  and $1\le k\le r-2$.  Then there exists a constant $C$ such that for sufficiently large $n$, $\Pr\left[|Z|\le C\left(\frac{u}{n}\right)^kn^{-k}\right]\le \exp\left(\Omega(n^{-\tau/r^2})\right)$
\end{replem}
\begin{proof}

First we use Lemma \ref{CountingLemma} to compute $\hat{Z}(V)$ in terms of $k$ and $c(V)$.  We do this in three pieces.  For $i\in \{0,1,2\}$ let
$\mathcal{F}_0$ be the family of nonempty subsets $V\subset\B$ such that $|\supp(V)\cap U_0|=i$.

For any $V\subset \B$ and $e\in {U_0\choose 2}$ note that $\widehat{G_e^2}(V)=0$ unless $\supp(V)\cap U_0\subset e$.  Therefore if $V\in \mathcal{F}_i$, then there are at most $n^{2-i}$ choices of $e\in {U_0\choose 2}$ such that $\widehat{G_e^2}(V)\neq 0$.  Combining this observation with Equation \ref{GeCoeffBound} and Lemma \ref{CountingLemma} for $V\in \mathcal{F}_i$ we have
\begin{align*}
\hat{Z}(V)=\sum_{e\in B_0}\widehat{G_e^2}(V)\le\O\left(n^{2-i}n^{-2k-2}u^{k-c(V)}\right)=O\left(\left(\frac{u}{n}\right)^{k-c(V)}n^{-k-c(V)-i}\right)
\end{align*}

%
%
%

Now we are in a position to compute $Var(Z)$.  We break the sets up into three parts as per the above calculation.  First
for sets $V\in \mathcal{F}_0$ we use Lemma \ref{CD} to say
\begin{align*}
\sum_{V\in \mathcal{F}_0}\hat{Z}(V)^2&=\sum_{c=1}^k\sum_{\substack{V\in \mathcal{F}_0\\c(V)=c}} \hat{Z}(V)^2
=\sum_{c=1}^k\sum_{\substack{V\in \mathcal{F}_0\\c(V)=c}} O\left(\left(\frac{u}{n}\right)^{2k-2c}n^{-2k-2c}\right)\\
&=\sum_{c=1}^k O\left(\left(\frac{u}{n}\right)^{2k-2c}n^{-2k-2c}u^{2c-1}\right)=O\left(\left(\frac{u}{n}\right)^{2k-1}n^{-2k-1}\right)
\end{align*}

Next for $\mathcal{F}_i$ where $i=1,2$ we find
\begin{align*}
\sum_{V\in \mathcal{F}_1}\hat{Z}(V)^2&=\sum_{c=1}^k\sum_{\substack{V\in \mathcal{F}_1\\c(V)=c}} \hat{Z}(V)^2
=\sum_{c=1}^k\sum_{\substack{V\in \mathcal{F}_1\\c(V)=c}} O\left(\left(\frac{u}{n}\right)^{2k-2c}n^{-2k-2c-2i}\right)\\
&=\sum_{c=1}^k O\left(\left(\frac{u}{n}\right)^{2k-2c}n^{-2k-2c-2i}u^{2c}n^{i}\right)=O\left(\left(\frac{u}{n}\right)^{2k}n^{-2k-i}\right)
\end{align*}
%
%
Combining these two bounds we find
$$Var(Z)=\sum_{\substack{V\subset \B\\V\neq \varnothing}} \hat{Z}(V)^2=\sum_{V\in \mathcal{F}_0} \hat{Z}(V)^2+\sum_{V\in \mathcal{F}_1} \hat{Z}(V)^2+\sum_{V\in \mathcal{F}_2} \hat{Z}(V)^2= O\left(\left(\frac{u}{n}\right)^{2k-1}n^{-2k-1}\right)$$

We know  from Lemma \ref{GHZ} that $\E[Z]=\Theta(u^kn^{-2k})$, so the event that $|Z|\le \E[Z]/2$ implies that 

$$\frac{|Z-\E[Z]|}{\sqrt{Var(Z)}}\ge \frac{\E[Z]}{2\sqrt{Var(Z)}}=\Omega\left(\frac{u^kn^{-2k}}{u^{k-\frac12}n^{-2k}}\right)=\Omega\left(u^{\frac12}\right)=\Omega\left(n^{\tau}\right)$$
Since $Z$ is a polynomial of degree at most $r(r-1)$, a standard application of Lemma \ref{concentration hypercontractivity} confirms
that $\Pr\left[|Z|\le \E[Z]/2\right]\le \exp\left(-\Omega(n^{2\tau/r^2})\right)$, and this is what we needed to show.
\end{proof}

\end{document}